\documentclass[11pt,twoside]{amsart}

\usepackage{geometry}
\usepackage{hyperref}
\usepackage{amsmath,amssymb,amsthm}
\usepackage[active]{srcltx}
\usepackage{multirow}
\usepackage{pgfplots}
\usepgfplotslibrary{fillbetween}
\usepackage{xcolor}

\title[Comparison principles for elliptic branches]{Comparison principles for viscosity solutions of elliptic branches of fully nonlinear equations independent of the gradient}
\author{Marco Cirant}
\address{Dipartimento di Scienze Matematiche Fisiche e Informatiche, Universit\`a di Parma, Parco Area delle Scienze 53/a, 43124--Parma, Italy}
\email{marcoalessandro.cirant@unipr.it (Marco Cirant)}\thanks{Cirant partially supported by the Fondazione CaRiPaRo Project ``Nonlinear Partial Differential Equations: Asymptotic Problems and Mean-Field Games'' and the Programme ``FIL-Quota Incentivante'' of University of Parma, co-sponsored by Fondazione Cariparma.}
\author{Kevin R.\ Payne}
\address{Dipartimento di Matematica ``F. Enriques''\\ Universit\`a di Milano\\ Via C. Saldini 50\\ 20133--Milano, Italy}
\email{kevin.payne@unimi.it (Kevin R. Payne)} \thanks{Payne partially supported by the Gruppo Nazionale per l'Analisi Matematica, la Probabilit\`a e le loro Applicazioni (GNAMPA) of the Istituto Nazionale di Alta Matematica (INdAM) and the projects: GNAMPA 2017 ``Viscosity solution methods for fully nonlinear degenerate elliptic equations'', GNAMPA 2018 ``Costanti critiche e problemi asintotici per equazioni completamente non lineari'' e GNAMPA 2019 ``Problemi differenziali per operatori fully nonlinear fortemente degeneri''.}
\date{\today} 


\subjclass[2010]{}

\catcode`@=11 \@addtoreset{equation}{section} \catcode`@=12

\newcommand{\cS}{\mathcal{S}}

\newcommand{\Q}{\mathcal{Q}}

\newcommand{\R}{\mathbb{R}}
\newcommand{\N}{\mathbb{N}}

\newcommand{\USC}{\mathrm{USC}}
\newcommand{\LSC}{\mathrm{LSC}}

\newcommand{\SA}{\mathrm{SA}}

\newcommand{\TSH}{\Theta \mathrm{SH}}
\newcommand{\TSHD}{\widetilde{\Theta} \mathrm{SH}}

\newcommand{\QSHD}{\widetilde{\mathcal{Q}} \mathrm{SH}}
\newcommand{\veps}{\varepsilon}
\newcommand{\cJ}{\mathcal{J}}
\newcommand{\cI}{\mathcal{I}}
\newcommand{\cF}{\mathcal{F}}
\newcommand{\cP}{\mathcal{P}}
\newcommand{\cN}{\mathcal{N}}
\newcommand{\cQ}{\mathcal{Q}}
\newcommand{\wt}{\widetilde}

\newtheorem{thm}{\textbf{Theorem}}[section]
\newtheorem{lem}[thm]{\textbf{Lemma}}
\newtheorem{prop}[thm]{\textbf{Proposition}}
\theoremstyle{remark}
\newtheorem{rem}[thm]{\textbf{Remark}}

\theoremstyle{definition}
\newtheorem{defn}[thm]{\textbf{Definition}}
\newtheoremstyle{Claim}{}{}{\itshape}{}{\itshape\bfseries}{:}{ }{#1}
\theoremstyle{Claim}

\newtheorem{ack}{Acknowledgment}

\begin{document}

\maketitle

\makeatletter
\def\l@subsection{\@tocline{2}{0pt}{2.5pc}{5pc}{}}
\makeatother

\setcounter{tocdepth}{1}
\tableofcontents

\begin{abstract} The validity of the comparison principle in variable coefficient fully nonlinear gradient free potential theory is examined and then used to prove the comparison principle for fully nonlinear partial differential equations which determine a suitable potential theory. The approach combines the notions of {\em proper elliptic branches} inspired by Krylov \cite{Kv95} with the  {\em monotonicity-duality method} initiated by  Harvey and Lawson \cite{HL09}. In the variable coefficient nonlinear potential theory, a special role is played by the {\em Hausdorff continuity} of the {\em proper elliptic map} $\Theta$ which defines the potential theory. In the applications to nonlinear equations defined by an operator $F$, structural conditions on $F$ will be determined for which there is a {\em correspondence principle} between {\em $\Theta$-subharmonics/superharmonics} and {\em admissible viscosity sub and supersolutions} of the nonlinear equation and for which comparison for the equation follows from the associated compatible potential theory. General results and explicit models of interest in differential geometry will be examined. Examples of improvements with respect to existing results on comparison principles will be given.
\end{abstract}

\section{Introduction}\label{sec:intro}

The main result of this paper concerns the validity of the comparison principle 
\begin{equation}\label{ulev}
\mbox{$u \leq v$ on $\partial \Omega \ \ \Rightarrow \ \ u \leq v$ in $\Omega$}
\end{equation}
when $u$ and $v$ are viscosity solutions of variable coefficient second order gradient free differential inclusions 
\begin{equation}\label{diffinc}
\mbox {$(u(x), D^2u(x)) \in \Theta(x)$ \quad and  \quad $ (v(x), D^2v(x)) \notin [\Theta(x)]^{\circ}$ \ \ for each $x \in \Omega$,} 
\end{equation}
where $\Omega \subset \R^N$ is a bounded domain and  $\Theta: \Omega \multimap \R \times \cS(N)$ is a suitable set valued map. We will show that the comparison principle holds if $\Theta$ is a {\em proper elliptic map} which is {\em Hausdorff continuous} (see Theorem \ref{thm:cpi} below). Proper ellipticity of the map means that $\Theta$ takes values in the {\em proper elliptic subsets} of $\R \times \cS(N)$. As formalized in Definition \ref{proper_ell_map}, this means that each 
$\Theta(x)$ is a non empty, closed, and proper subset of $\R \times \cS(N)$ which is {\em $\cQ$-monotone}; that is, 
\begin{equation}\label{Qmono}
\Theta(x) + \cQ \subset \Theta(x) \ \ \text{where} \ \cQ := \{(s,A) \in \R \times \cS(N): \ s \leq 0, A \geq 0 \} := \cN \times \cP,
\end{equation}
where $\cS(N)$ denotes the space of symmetric $N \times N$ matrices with its natural partial ordering. Hausdorff continuity just means that $\Theta$ is continuous as a map from Euclidian space into the metric space of closed subsets of $\R \times \cS(N)$ equipped with the Hausdorff distance (see Definition \ref{defn:HC} and the remarks which follow). For proper elliptic maps, we will show in Proposition \ref{UCHD} that this (locally uniform) continuity is equivalent to the statement that: for each $\Omega' \subset\subset \Omega$ and $\eta > 0$, there exists $\delta = \delta(\eta, \Omega') > 0$ such that for each $x,y \in \Omega'$
\begin{equation}\label{HC:intro}
\mbox{$|x -y| < \delta \ \Rightarrow \ \Theta(x) + (- \eta, \eta I) \subset \Theta(y)$.}
\end{equation}

Our main result will then be used to establish the validity of the comparison principle for second order gradient-free fully nonlinear PDEs of the form
\begin{equation}\label{FNLNG}
F(x, u(x), D^2 u(x)) = 0, \ \ x \in \Omega,
\end{equation}
where $F$ is a continuous function of its arguments. The equations \eqref{FNLNG} we treat will be {\em proper elliptic} in a sense which is inspired by Krylov's general notion of ellipticity \cite{Kv95}. More precisely, one shifts attention from the equation \eqref{FNLNG} to
the differential inclusion
\begin{equation}\label{DIF}
     \mbox{$(u(x), D^2u(x)) \in \partial \Theta(x)$ for each $x \in \Omega$,}
\end{equation}
where one requires that
 \begin{equation}\label{branch}
    \mbox{$ \partial \Theta(x) \subset \Gamma(x) := \{ (r,A) \in \R \times \cS(N): \ F(x,r,A) = 0 \}$ for each $x \in \Omega$.}
\end{equation}
If $\Theta$ is a {\em proper elliptic map}, then the inclusion \eqref{branch} is called a {\em proper elliptic branch} of the equation \eqref{FNLNG}. We will give sufficient conditions on the operator $F$ which ensure the existence of a continuous proper elliptic map $\Theta$ such that the {\em branch condition} \eqref{branch} holds and for which viscosity solutions $u,v$ of the differential inclusion \eqref{diffinc} correspond to {\em admissible viscosity subsolutions, supersolutions} of the differential equation. In this way, the comparison principle for the equation \eqref{FNLNG} follows from the validity of the comparison principle for \eqref{diffinc}. 

In order to carry out this program, we will treat operators $F$ which are {\em proper elliptic} in the sense that
\begin{equation}\label{PEP1}
F(x,r,A) \leq F(x,r + s, A + P) \ \ \text{for each} \ (r,A) \in \Phi(x), \ (s,P) \in \cQ = \cN \times \cP;
\end{equation}
where either 
	\begin{equation}\label{CPhi_Intro}
\mbox{$\Phi$ is a Hausdorff continuous proper elliptic map on $\Omega$ \ \ (constrained case)}
\end{equation}
or
\begin{equation}\label{UCPhi_Intro}
\mbox{$\Phi(x) = \cJ = \R \times \cS(N)$ for each $x \in \Omega$ \ \ (unconstrained case).}
\end{equation}
The notion of admissibility mentioned above uses $\Phi$ as an additional constraint on the test functions (or test jets) used in the viscosity formulation of subsolutions and supersolutions. In the unconstrained case \eqref{UCPhi_Intro}, this constraint is silent and one recovers the classical viscosity formulations. On the other hand, the constrained case \eqref{CPhi_Intro} arises in situations where $F(x, \cdot, \cdot)$ is suitably monotone only on proper subsets $\Phi(x) \subsetneq \R \times \cS(N)$, and the admissibility constraint is essential. Treating such situations in a general and coherent way is a main motivation of the present work.

\medskip

Before giving additional details of the possible applications to PDEs, we discuss the comparison principle \eqref{ulev} in the framework of nonlinear potential theory: given a {\em $\Theta$-subharmonic function} $u$ and a {\em $\Theta$-superharmonic function} $v$; that is,  upper semicontinuous and lower semicontinuous functions $u$ and $v$ on $\overline \Omega$ satisfying \eqref{diffinc} in a viscosity sense (see Definition \ref{defn:TSH}), we look for monotonicity and regularity properties on $\Theta$ that guarantee the validity of \eqref{ulev}. This program was initiated systematically in the groundbreaking work of Harvey and Lawson \cite{HL09} for differential inclusions
$$
\mbox {$D^2u(x) \in \Theta$ \quad and  \quad $ D^2v(x) \notin \Theta^{\circ}$ \ \ for each $x \in \Omega$;} 
$$
that is, in the context of a pure second order constant coefficient potential theory in which the elliptic map appearing in the inclusion does not depend on the $x$ variable. Such a potential theory might come from a purely second order operator of interest. In this situation, which corresponds to $\Theta(x) = \R \times \Theta$ on $\Omega$ in \eqref{diffinc}, it has been proven that the mere monotonicity assumption $\Theta + P \subset \Theta$ for all $P \ge 0$ is sufficient for the comparison principle to hold. The introduction of a genuine $x$-dependence and further constraints on $(u(x), Du(x))$ poses then the natural question of what are minimal conditions on the map $\Theta$ that guarantee the validity of the comparison principle. An important step in this program has been settled in a subsequent work \cite{HL11} for general differential inclusions involving the full 2-jet of $u$ on Riemannian manifolds; the approach in \cite{HL11} is based on the reduction of the $x$-dependent case to a constant map (called constant coefficient subequation) of the form $(u(x), Du(x), D^2 u(x)) \in \mathcal F \subset \R \times \R^N \times \cS(N)$, for which a general theory is developed. We mention that in the case of constant maps $\mathcal F$ on Euclidean space, further generalizations in the direction of minimal monotonicity assumptions on $\mathcal F$ are a work in progress \cite{CHLP19}. While the general approach of \cite{HL11} covers a wide variety of situations, it requires implicit assumptions on the $x$-dependence in the problem. In partial contrast, \cite{CP17} has been devoted to the search for more explicit conditions in the special case of $(u, Du)$-independent inclusions $D^2 u(x) \in \Psi(x) \subset \cS(N)$, with a particular focus placed on the minimal assumptions on the set-valued map $\Psi$ needed for the comparison principle. A main aim of the present work is to generalize results of \cite{CP17} to $Du$-independent inclusions of the form \eqref{diffinc}, which also allow for constraints on $u(x)$, and again with the purpose of identifying {\it monotonicity} and {\em regularity} properties of $\Theta$ as a set-valued map that lead to comparison principles. We finally mention that a recent work \cite{HL18} addresses similar (and additional) issues for maps of the form $\mathcal F(x) = \{(u, p, A) : F(u, p, A) \ge f(x)\}$, namely for maps $\mathcal F$ that are given by superlevel sets of a proper elliptic operator. For gradient-free operators, results obtained in \cite{HL18} fit into our theory.

We now explain how we aim to prove the comparison principle \eqref{ulev} for a variable coefficient gradient-free potential theory determined by $\Theta$. We follow the approach initiated in \cite{HL09} as continued in \cite{CP17}. There are three main ingredients in this approach: {\em monotonicity, duality} and {\em continuity}. The natural monotonicity in this context is to require that $\Theta(x)$ in $\cQ$-monotone in the sense \eqref{Qmono}. For the PDE applications, when $\Theta$ is suitably associated to a differential operator $F= F(x,r,A)$, this $\cQ$-monotonicity reflects the typical monotonicity properties of properness (decreasing in $r$) and degenerate ellipticity (increasing in $A$) for $F$. The natural notion of duality involves the {\em Dirichet dual} $\wt{\Theta}$ of $\Theta$, defined pointwise by 
\begin{equation}\label{DD}
\widetilde \Theta(x) := - \left[\Theta(x)^{\circ}\right]^c, \ \ x \in \Omega
\end{equation}
and introduced by Harvey and Lawson in \cite{HL09}. The dual map $\wt{\Theta}$ is proper elliptic if and only if $\Theta$ is. Moreover, $v$ is $\Theta$-superharmonic if and only if $\tilde{u}=-v$ is $\wt{\Theta}$-subharmonic. The first step in the monotonicity-duality approach for comparison \eqref{ulev} is to prove the relevant {\it subharmonic addition theorem}, which in this setting means (see Theorem \ref{thm:AAT}): given a $\Theta$-subharmonic function $u$ and a $\Theta$-superharmonic function $v$
\begin{equation}\label{AAT:intro}
\mbox{$u$ is $\Theta$-subharmonic, \ $\tilde u = -v$ is $\widetilde \Theta$-subharmonic \ $\Rightarrow$ \ $u + \tilde u$ is $\widetilde \cQ$-subharmonic}.
\end{equation}
It is worth noting that $\wt{\cQ}$ is a constant map, even when $\Theta$ and $\wt{\Theta}$ are not. The subharmonic addition theorem reduces the comparison principle \eqref{ulev} to the validity of the {\em zero maximum principle} (see Theorem \ref{thm:ZMP}): for every $w$ which is $\wt{\cQ}$-subharmonic on $\Omega$
\begin{equation}\label{ZMP_intro}
w \leq 0 \ \ \text{on} \ \partial \Omega \ \ \Rightarrow \ \ w \leq 0 \ \ \text{in} \ \Omega.
\end{equation}
Our proof of \eqref{ZMP_intro} exploits the following characterization of $\widetilde \cQ$-subharmonics as those functions $w$ whose positive part satisfies a comparison principle with respect to affine functions $a$: for all open subsets $X$ of $\Omega$,
\begin{equation}\label{SAP_intro}
\mbox{$w^+ \leq a$ on $\partial X \ \ \Rightarrow \ \ w^+ \leq a$ on $X$.}
\end{equation}

The proof of the fundamental subharmonic addition theorem \eqref{AAT:intro} relies on a reduction to semi-convex functions. To perform this reduction, based on sup-convolution approximations, one needs some control on how the proper elliptic sets $\Theta(x)$ behave as $x$ varies in $\Omega$. In particular, one needs to control the {\it distance} between $\Theta(x)$ and $\Theta(y)$ as subsets of $\R \times \cS(N)$. We will prove that a sufficient regularity condition is to requires the {\it Hausdorff continuity} of $\Theta$ taking values in the closed subsets of $\R \times \cS(N)$, thus generalizing the analogous condition in \cite{CP17} for maps with values in the closed subsets of $\cS(N)$. Once the reduction to semi-convex functions is available, the subharmonic addition theorem is obtained by means of Dirichlet duality and a Jensen-type lemma on the passage of almost everywhere to everywhere information (see Lemma \ref{lem:Slod}). 
Our main comparison result, generalizing the one in \cite{CP17} for set-valued maps in $\cS(N)$, is the following result (see Theorem \ref{thm:CP} for the proof).

\begin{thm}[Comparison principle: potential theoretic version]\label{thm:cpi} Let $\Theta$ be a Hausdorff continuous proper elliptic map on $\Omega$. Then the comparison principle holds; that is,
if $u \in \USC(\overline{\Omega})$ and $v \in \LSC(\overline{\Omega})$ are $\Theta$-subharmonic and $\Theta$-superharmonic respectively in $\Omega$ (in the sense of Definition \ref{defn:TSH}), then
$$
\mbox{$u \leq v$ on $\partial \Omega \ \ \Rightarrow \ \ u \leq v$ in $\Omega$.}
$$
\end{thm}

\medskip

We now return to the discussion of some possible applications of our potential theoretic result to fully nonlinear PDE. There is an extensive literature for treating general fully nonlinear elliptic equations via viscosity methods, and several attempts to restate or relax standard structural conditions (such as those stated in \cite{CIL92}) have been proposed. For example, when the equation lacks of strict monotonicity in the $u$-variable, as in $u$-independent equations, one can rely on some strict monotonicity (in some direction) with respect to the Hessian variable, see e.g. \cite{BM06, BB01, KK07}. As noted above, we are particularly interested in the constrained case where $F(x, \cdot, \cdot)$ is proper elliptic only when restricted to some admissibility constraint set $\Phi(x) \subsetneq \R \times \cS(N)$. By exploiting Krylov's idea of shifting the focus to the level sets of the operator $F$, the potential theoretic approach of Harvey and Lawson furnishes an elegant and unified framework to treat viscosity solutions with admissibility constraints in many situations that would otherwise require ad-hoc adjustments for a given operator $F$ of interest.

In the constrained case, a first general application to PDEs based on the comparison principle for differential inclusions developed here, is a comparison principle for the equation \eqref{FNLNG} under the following assumptions on $F$: there exists a proper elliptic map $\Phi: \Omega \to \wp(\R \times \cS(N))$ such that for each $x \in \Omega$ one has
\begin{equation}\label{PEP}
F(x,r,A) \leq F(x,r + s, A + P) \ \ \text{for each} \ (r,A) \in \Phi(x), \ (s,P) \in \cQ = \cN \times \cP;
\end{equation}
\begin{equation}\label{NEC}
\mbox{there exists $(r,A) \in \Phi(x)$ such that $F(x,r,A) = 0$;} 
\end{equation}
\begin{equation}\label{BC}
\mbox{ $\partial \Phi(x) \subset  \{ (r,A) \in \R \times \cS(N):  F(x,r,A) \leq 0 \}$ for each $x \in \Omega$;} 
\end{equation}
and for each $\Omega' \subset \subset \Omega$ and each $\eta > 0$ there exists $\delta = \delta(\eta, \Omega')$ such that 
\begin{equation}\label{RC}
\mbox{ $F(y, r - \eta, A + \eta I) \geq F(x, r, A) \quad \forall (r, A) \in \Phi(x), \forall x,y \in \Omega' \ {\rm with \ } |x - y| < \delta$. }
\end{equation}
The condition \eqref{PEP} states that $F(x, \cdot, \cdot)$ is proper elliptic if restricted to $\Phi(x)$. The conditions \eqref{NEC} and \eqref{BC} guarantee that the constraint $\Phi(x)$ is compatible with the zero locus of the operator $F$, so that the map defined by 
\begin{equation}\label{def_Theta}
\Theta(x) := \{ (r,A) \in \Phi(x):  \ F(x,r,A) \geq 0 \}
\end{equation}
is proper elliptic and defines a proper elliptic branch of \eqref{FNLNG} (see Theorem \ref{pick_branch}). The condition \eqref{RC} is a sufficient condition for $\Theta$ to be a Hausdorff continuous proper elliptic map, so that the comparison principle of Theorem \ref{thm:cpi} for $\Theta$-sub/superharmonic functions can be applied (see Theorem \ref{UCbranch}). One then obtains the comparison principle for {\it admissible} viscosity solutions to the PDE (see Definition \ref{Vs_def}), provided that a mild non-degeneracy assumption on $F$  holds (see formula \eqref{NDC} in the {\em correspondence principle} of Theorem \ref{thm:SHCVS}). The resulting comparison principle is stated in the main Theorem \ref{thm:CPpde}, which also covers the unconstrained case where $\Phi(x) = \R \times \cS(N)$ for all $x \in \Omega$.  

\begin{rem}\label{rem:improvements} Our main structural condition \eqref{RC} reflects a precise geometrical property of the associated map $\Theta$, which, in is some cases, is weaker than the general classical conditions in \cite{CIL92} (as will be noted for the equation \eqref{GMAE} below). Such improvements using our method were also seen for the reduced class of equations $F(x, D^2u) = 0$ (see  Remark 5.1 of \cite{CP17}). The condition \eqref{RC} can be regarded as a joint strict monotonicity with respect to $(r, A)$ and regularity with respect to the $x$ variable. On the other hand, the structural condition \eqref{RC} is \underline{not} necessary for $\Theta$ to be continuous.  Moreover, in some cases it may be easier to check directly the continuity of $\Theta$ by using Remark \ref{rem:cont0}.  An important example where this occurs is given below in \eqref{SLPE}.
\end{rem}

We now discuss various illustrations of our approach, in both constrained and unconstrained cases, where we will apply our method to two interesting model equations and their generalizations. An interesting feature of the methods presented here is how both situations can be placed into the same general framework. 

In the constrained case, we will prove a comparison principle for equations arising in the study of {\em hyperbolic affine hyperspheres}. The relevant equation can be written in the form
\begin{equation}\label{cy_intro}
[-u(x)]^{N+2} \det D^2 u(x) = h(x), \qquad x \in \Omega, \\
\end{equation}
where $h \geq 0$ is the negative of the curvature (when constant). These equations are proper elliptic on $\Phi := \cQ$ and are particularly degenerate (a lack of strict monotonicity properties) for vanishing curvatures. The comparison principle for \eqref{cy_intro} is given in Theorem \ref{thm:HASE}. This result is then generalized in  Theorem \ref{thm:CCCP} to the following class of perturbed Monge-Amp{\`e}re equations
\begin{equation}\label{GMAE}
g(m(x) - u(x)) \det (D^2 u(x) + M(x)) = h(x),
\end{equation}
where $g,m,M$ are continuous functions, and $g(\cdot)$ is increasing and positive on some open interval $(r_0, \infty)$. For a perturbation matrix $M$ which is merely continuous, the equation \eqref{GMAE} does \underline{not}, in general, satisfy the standard structural condition (3.14) of Crandall-Ishii-Lions \cite{CIL92}. 

In the unconstrained case, we present a new comparison principle for the {\em special Lagrangian potential equation}
\begin{equation}\label{SLPE}
\sum_{i = 1}^N \arctan \big( \lambda_i (D^2 u(x)) \big) = h(x),
\end{equation}
where $\{\lambda_i (A) \}_{i = 1}^N$ are the eigenvalues of $A \in \cS(N)$ and the {\em phase} $h$ takes values in the interval $\cI := (-N \pi/2, N \pi/2)$. This equation for $h$ fixed is proper elliptic (but possibly highly degenerate) on all of $\Phi:= \R \times \cS(N)$. The equation \eqref{SLPE} with constant phases $h(x) \equiv \theta$ was introduced by Harvey-Lawson \cite{HL82} in the study of {\em calibrated geometries} and existence and uniqueness of viscosity solutions in this case is known from their work (see \cite{HL09} and \cite{HL19}). The inhomogeneous equation also has a natural geometric interpretation (see, for example the discussion in \cite{HL19}), but it is less well understood. A key feature in the theory is played by the {\em special phase values}
\begin{equation}\label{SPV}
 \theta_k:= (N - 2k)\pi/2 \ \ \text{for} \  k = 1, \ldots N - 1,
 \end{equation}
which determine the {\em phase intervals}
\begin{equation}\label{I_k}
	\cI_k:= \left( (N - 2k) \frac{\pi}{2}, (N - 2(k - 1)) \frac{\pi}{2} \right) \ \ \text{with} \ \ k = 1, \ldots N.
\end{equation}
In Theorem \ref{thm:CSLPE}, we show that the comparison principle holds if $h$ is continuous and takes values in any one of the phase intervals \eqref{I_k}; that is, if 
 \begin{equation}\label{phase_condition}
 h(\Omega) \subset \cI_k \ \ \text{for any fixed} \ k = 1, \ldots, N.
 \end{equation} 
Our proof involves a delicate argument to show that the natural proper elliptic map
 \begin{equation}\label{SLP_PEM}
 \Theta(x):= \left\{ (r,A) \in \R \times \cS(N): \  \sum_{i = 1}^N \arctan \big( \lambda_i (A) \big) - h(x) \geq 0 \right\}, \ \ x \in \Omega
 \end{equation}
is Hausdorff continuous if \eqref{phase_condition} holds and hence the comparison principle for \eqref{SLPE} follows from Theorem \ref{thm:cpi}.

It is important to note that the comparison principle (and much more) is known in the special case for $h$ taking values in the top phase interval $\cI_1 = ((N-2)\pi/2, N\pi/2)$ as shown in Dinew, Do and T\^{o} \cite{DDT19}. An alternate proof is given by Harvey and Lawson \cite{HL19} which makes use of the notion of {\em tameness} of the operator $G(A):= \sum_{i = 1}^N \arctan \big( \lambda_i (A) \big)$ on the subequation $\mathcal{G}(\theta) := \{A \in \cS(N): G(A) \geq \theta\}$).

 It is also important to note that while the operator
 \begin{equation}\label{SLPO}
 F(x,r,A):=  \sum_{i = 1}^N \arctan \big( \lambda_i (A) \big) - h(x)
 \end{equation}
is  proper elliptic on all of $\Phi= \R \times \cS(N)$, it \underline{fails} to satisfy our regularity condition \eqref{RC} if $h$ takes on any of the special values $\theta_k$ in \eqref{SPV}. Indeed, we will show that the proper elliptic map $\Theta$ in \eqref{SLP_PEM} fails to be Hausdorff continuous if any continuous (and non constant) $h$ takes on any of the special values (see Proposition \ref{prop:failure}). This leaves open the question whether comparison also holds for continuous $h$ which takes on a special phase value (see Open Question on page 23 of \cite{HL19}).

We have focused attention on proper elliptic pairs $(F, \Theta)$ which are {\em compatible} in the sense that (see Remark \ref{rem:CEP}):
\begin{equation}\label{compatible}
	\partial \Theta(x) = \{ (r,A) \in \Theta(x): F(x,r,A) = 0 \} \neq \emptyset, \ \ \text{for each} \ x \in \Omega.
\end{equation}
This ensures the correspondence between $\Theta$ subharmonics/superharmonics and admissible viscosity subsolutions/supersolutions of the equation determined by the operator $F$. Hence, given $F$ one can pass to the potential theory determined by $\Theta$ and then ``come back'' to the admissible viscosity formulation for the operator $F$. However, in situations in which $(F, \Theta)$ are a proper elliptic pair, but the compatibility \eqref{compatible} fails, one could decide to use the potential theory determined by $\Theta$ as a replacement for a viscosity solution treatment of the equation. In the constrained case, compatibility fails if the {\em non-degeneracy} condition \eqref{NDC} fails.

As a final introductory remark, we have limited the present investigation to the validity of the comparison principle \eqref{ulev}. Our comparison principles would yield uniqueness results for the Dirichlet problem on $\Omega$ in both the PDE and potential theoretic settings. In particular, we leave the important and interesting question of existence for the Dirichlet problem for a future work. An important feature of the methods pioneered by Harvey and Lawson is the determination (in terms of $\Theta$) of the suitable boundary convexity needed to obtain existence. The reader might wish to consult \cite{HL09}, \cite{HL11}, \cite{CP17} and \cite{HL18} for the use of a Perron method for existence in many situations which would cover some of the equations and potential theories considered here. Finally, the use of viscosity solutions with admissibility constraints has been extended to include some elements of nonlinear spectral theory in \cite{BP19}, including characterization of principal eigenvalues and existence of associated principal eigenfunctions by maximum principle methods.

\section{Proper elliptic maps and their subharmonics}\label{sec:PEMs}

In all that follows, $\Omega \subset \subset \R^N$ will be a bounded open connected set and  $\mathcal{S}(N)$ will denote the space of symmetric $N \times N$ matrices, which carries the usual partial ordering of the associated quadratic forms and $\lambda_1(A) \leq \cdots \leq \lambda_N(A)$ denote the ordered eigenvalues of $A \in \cS(N)$. We will denote by $\wp(\R \times \cS(N)):= \{ \Phi: \Phi \subset \R \times \cS(N) \}$ and use the notations $\overline{\Phi}, \Phi^{\circ}$ and $\Phi^c$ for the closure, interior and complement of $\Phi \in \wp(\R \times \cS(N))$. We will also make use of spaces of semicontinuous functions
$$
    \USC(\Omega) = \{ u: \Omega \to [-\infty, \infty): \ u(x_0) \geq \limsup_{x \to x_0} u(x),\ \forall \ x_0 \in \Omega\}
$$
and
$$
    \LSC(\Omega) = \{ u: \Omega \to (-\infty, \infty]: \ u(x_0) \leq \liminf_{x \to x_0} u(x),\ \forall \ x_0 \in \Omega\}.
$$

\subsection{Proper elliptic maps and their duals}
We begin with the definition of the class of set valued maps we will use, where we denote by
\begin{equation}\label{P}
\mathcal{P} := \{ P \in \cS(N): \ P \geq 0 \} = \{ P \in \cS(N): \ \lambda_1(P) \geq 0 \},
\end{equation}
\begin{equation}\label{N}
\mathcal{N} := \{ s \in \R: \ s \leq 0 \},
\end{equation}
and
\begin{equation}\label{Q}
\Q := \mathcal{N} \times \mathcal{P} = \{ (s,P) \in \R \times \cS(N): \ s \leq 0 \ \ \text{and} \ \  P \geq 0 \}.
\end{equation}

\begin{defn}\label{proper_ell_map} A map $\Theta: \Omega \to \wp(\R \times\cS(N))$ is said to be a {\em proper elliptic map} if for each $x \in \Omega$, one has
\begin{equation}\label{PE1}
\mbox{ $\Theta(x)$ is a closed, non empty and proper subset of $\R \times\cS(N)$}
\end{equation}
and 
\begin{equation}\label{PE2}
\Theta(x) + \Q \subset \Theta(x);
\end{equation}
that is, if $(r,A) \in \Theta$ the $(r + s, A + P) \in \Theta(x)$ for each $s \leq 0$ and $P \geq 0$. We will also say that $\Theta(x)$ is {\em $\cQ$-monotone} if \eqref{PE2} holds. 
\end{defn}

If $\mbox{$\mathcal{E}:= \left\{ \Phi \subset \R \times \cS(N): \Phi \right.$ is closed, non empty and proper with $ \left. \Phi + \Q \subset \Phi \right\}$ }$, then a proper elliptic map is just a set valued map taking values in $\mathcal{E}$, the collection of {\em proper elliptic sets}. Note that a proper elliptic map is {\em strict} as a set-valued map, namely it satisfies $\Theta(x) \neq \emptyset$ for each $x \in \Omega$ (see Chapter 1 of Aubin and Cellina \cite{AC84} for the elementary notions concerning set-valued maps).

An important example is provided by the constant map $\Theta(x) = \Q$ for all $x \in \Omega$ with $\Q$ defined by \eqref{Q}. Clearly the $\cQ$-monotonicity \eqref{PE2} is related to the monotonicity properties of $F(x,r,A)$ for proper and degenerate elliptic equations \eqref{FNLNG}. Notice that if $\Theta = \Theta(x,A)$ independent of $r \in \R$ then we can identify $\Theta$ with an {\em elliptic map} in the sense of \cite{CP17} and if $\Theta = \Theta(A)$ is also independent of $x$ we can identify $\Theta$ with an {\em elliptic set (Dirichlet set)} in the sense of \cite{HL09}.

A class of {\em dual maps} using the the Dirichlet dual, introduced by Harvey-Lawson \cite{HL09}, plays an essential role in this theory. 
\begin{defn}\label{dual_map} Let $\Theta: \Omega \to \wp(\R \times\cS(N))$ be a proper elliptic map. The {\em dual map} $\widetilde{\Theta}: \Omega \to \wp(\R \times\cS(N))$ is defined pointwise by
\begin{equation}\label{DM}
\widetilde{\Theta}(x) = \left[-\Theta(x)^{\circ}\right]^c = - \left[\Theta(x)^{\circ}\right]^c.
\end{equation}
\end{defn}
An essential example is given by the dual to constant map $\Q$, which is the constant map $\widetilde{\Theta}(x) = \widetilde{\Q}$ for each $x \in \Omega$ where
\begin{equation}\label{Q_dual}
\widetilde{\Q} = \{ (r,A) \in \R \times \cS(N): r \leq 0 \ \ \text{or} \ \ A \in \widetilde{\mathcal{P}} \}
\end{equation}
and
\begin{equation}\label{P_dual}
\widetilde{\mathcal{P}} = \{ (A \in  \cS(N): \lambda_N(A) \geq 0 \},
\end{equation}
as a simple calculation shows. We record the following elementary properties which will be used throughout.

\begin{prop}\label{Elem_Props} Let $\Theta: \Omega \to \wp(\R \times\cS(N))$ be a proper elliptic map. Then the following properties hold.
\begin{itemize}
\item[(a)] The dual $\widetilde{\Theta}: \Omega \to \wp(\R \times\cS(N))$ is a proper elliptic map. Moreover an arbitrary map $\Theta$ will be a proper elliptic map if its dual map is.
\item[(b)] The dual of $\widetilde{\Theta}$ is the map $\Theta$.
\item[(c)] The sum of $\Theta$ and $\widetilde{\Theta}$ satisfies
\begin{equation}\label{sum_of_duals}
    \Theta(x) +  \widetilde{\Theta}(x) \subset \widetilde{\cQ}, \quad \text{for each\ } x \in \Omega.
\end{equation}
    Moreover, $(r,A) \in \Theta(x)$ if and only if $(r+s,A+B) \in \widetilde{\Q}$ for each $(s,B) \in \widetilde{\Theta}(x)$.
\item[(d)] For each $x \in \Omega$ one has
\begin{equation}\label{bdy_theta}
    \partial \Theta(x) = \Theta(x) \cap  \left( - \widetilde{\Theta}(x) \right).
\end{equation}
\item[(e)] For each $x \in \Omega$ one has
\begin{equation}\label{closure_interior}
\Theta(x) = \overline{\Theta(x)^{\circ}}
\end{equation}
\end{itemize}
\end{prop}

\begin{proof} The claims (b) and (d) follow directly from the pointwise definition of the dual map \eqref{DM}. For the claim (e), recall that by definition the set $\Theta(x)$ is closed in the natural topology of $\R \times \cS(N)$ for each $x \in \Omega$, where $\Theta(x)$ has non empty interior. Each $(r,A) \in \Theta(x)$ can be written as the limit as $\veps \to 0^+$ of $(r - \veps , A + \veps I) \in \Theta(x) + (\mathcal{N}^{\circ} \times \mathcal{P}^{\circ})$, and $ \Theta(x) + (\mathcal{N}^{\circ} \times \mathcal{P}^{\circ}) = [\Theta(x)]^\circ$. The claims (a) and (c) make use of various known identities for elliptic sets and elliptic maps as presented in \cite{HL09} and \cite{CP17}.
\end{proof}

\subsection{Weakly subharmonic functions associated to proper elliptic maps}

The main concept in this paper concerns the {\em $\Theta$-subharmonic functions} on a domain $\Omega$ which are determined by a proper elliptic map $\Theta$ on $\Omega$. These upper semicontinuous functions are defined in a pointwise and viscosity sense by requiring that the relevant second order subdifferential lies in $\Theta(x)$. Following the approach of Harvey and Lawson, the natural class of {\em $\Theta$-superharmonic functions} will be characterized in terms of the $\wt{\Theta}$-subharmonic functions with respect to the dual map (see Definition \ref{defn:TSH} and Remark \ref{rem:Tsuper}). For $u$ twice differentiable, to be {\em $\Theta$-subharmonic on $\Omega$} means that 
\begin{equation}\label{ThetaSH_class}
J_xu:= (u(x), D^2 u(x)) \in \Theta(x), \quad \forall x \in \Omega;
\end{equation}
that is; that the (reduced) 2-jet $J_x u$ lies in the {\em constraint set} $\Theta(x)$ for each $x \in \Omega$. We will say that $u$ is {\em strictly $\Theta$-subharmonic in $\Omega$} if $J_xu \in  [\Theta(x)]^{\circ}$ for all $x \in \Omega$. For $u \in \USC(\Omega)$, one makes use of a viscosity definition. To this end, for each fixed $x_0 \in \Omega$, consider the {\em upper test jets}
\begin{equation}\label{superjets}
J^{+}_{x_0} u := \{ (\varphi(x_0),D^2 \varphi(x_0)):  \varphi \ \text{is} \ C^2 \ \text{near} \ x_0, \  u \leq \varphi \ \text{near $x_0$ with equality in} \ x_0 \}
\end{equation}
and the {\em lower test jets}
\begin{equation}\label{subjets}
J^{-}_{x_0} u := \{ (\varphi(x_0),D^2 \varphi(x_0)):  \varphi \ \text{is} \ C^2 \ \text{near} \ x_0, \  u \geq \varphi \ \text{near $x_0$ with equality in} \ x_0 \}.
\end{equation}

\begin{defn}\label{defn:TSH} Let $\Theta$ be a proper elliptic map on $\Omega$ and $x_0 \in \Omega$.
	\begin{itemize}
	\item[(a)] A function $u \in \USC(\Omega)$ will be called {\em $\Theta$-subharmonic in $x_0$} if
\begin{equation}\label{Tsubharm}
			J^{+}_{x_0}u \subset \Theta(x_0),
\end{equation}
	and $u$ is said to be {\em $\Theta$-subharmonic in $\Omega$} if \eqref{Tsubharm} holds for each $x_0$. The spaces of all such functions will be denoted by  $\TSH(x_0)$ and $\TSH(\Omega)$ respectively.
	\item[(b)] A function $u \in \LSC(\Omega)$ will be called {\em $\Theta$-superharmonic in $x_0$} if
\begin{equation}\label{Tsuperharm}
J^{-}_{x_0}u \subset \left[\Theta(x_0)\right]^{\circ}]^c.
\end{equation}
	\item[(c)] A function $u \in C(\Omega)$ will be called {\em $\Theta$-harmonic in $\Omega$} if it is both $\Theta$-subharmonic and $\Theta$-superharmonic in $\Omega$.
\end{itemize}
	\end{defn}
A few remarks about Definition \ref{defn:TSH} are in order.

\begin{rem}\label{rem:reduction} In the differential inclusion \eqref{Tsubharm} there is no constraint made in the gradient variable, which corresponds to the gradient free equations that we treat here. If one denotes by $\cJ^2 = \R \times \R^N \times \cS(N)$ by the space of 2-jets with jet coordinates $J=(r,p,A) \in \cJ^2$, then the inclusion \eqref{Tsubharm} is equivalent to
\begin{equation}\label{FSH}
	J^{2,+}_{x_0}u \in \cF(x_0)
\end{equation}
with a constraint set
\begin{equation}\label{subequation}
\cF(x_0) := \{ (r,p,A) \in \cJ^2: \ (r,A) \in \Theta(x_0) \}
\end{equation}
that is a subset of the (full) 2-jet space and where
$$
J^{2,+}_{x_0} u := \{ (\varphi(x_0), D\varphi(x_0), D^2 \varphi(x_0)):  \varphi \ \text{is} \ C^2 \ \text{near} \ x_0, \  u \leq \varphi \ \text{near $x_0$ with equality in} \ x_0 \}.
$$
is the set of second order superjets. The reduced formulation \eqref{Tsubharm} will be used throughout to simplify notation and to emphasize the gradient independent nature of the equations we consider.
	\end{rem}

\begin{rem}\label{rem:concact_jets}
	Many equivalent choices for the {\em upper/lower test functions} $\varphi$ which compete in \eqref{superjets}, \eqref{subjets} could be used in Definition \ref{defn:TSH}. For example, one could use {\em upper test jets} $J^{+}_{x_0} u$ corresponding to $\varphi = Q$ a quadratic polynomial. One could also assume that $(u - Q)$ has a strict maximum (of zero) in $x_0$ where for some $\veps > 0$ 
	\begin{equation}\label{TJ1}
		\mbox{$(u - Q)(x) \leq -\veps|x - x_0|^2$ \ for each $x$ near $x_0$ with equality at $x_0$},
			\end{equation}    
			or assume that
			\begin{equation}\label{TJ2}
			\mbox{$(u - Q)(x) \leq o(x - x_0)^2$ \ for each $x$ near $x_0$ with equality at $x_0$.}
				\end{equation}
				In all cases, the resulting spaces $\TSH(x_0)$ and $\TSH(\Omega)$ remain the same. For a proof, see Lemma A.1 of \cite{CHLP19} which treats general pointwise inclusions in the form \eqref{FSH}. 		
\end{rem}

\begin{rem}\label{rem:Tsuper}
 Since $[\Theta(x_0)^{\circ}]^c = - \widetilde{\Theta}(x_0)$ and $J^{+}_{x_0}(-u) = -J^{-}_{x_0}u$, one has
\begin{equation}\label{VSEM3}
\mbox{$u \in \LSC(\Omega)$ is $\Theta$-superharmonic in $x_0$ if and only if $-u \in \TSHD(x_0)$;}
\end{equation}
that is, $\Theta$-superharmonicity can be expressed in terms of subharmonicity for the dual map.
\end{rem} 


\begin{rem}\label{coherence} The following {\em coherence principle} for the classical and weak notions of $\Theta$-subharmonicity holds: let $u \in \USC(\Omega)$ be twice differentiable in $x_0 \in \Omega$ \footnote{ $u(x)=u(x_0) + \langle p, x-x_0 \rangle + \frac{1}{2} \langle A(x-x_0), x-x_0 \rangle + o(|x-x_0|^2)$ as $x \to x_0$ for some $(p,A) \in \R^N \times \cS(N)$. Hence $u$ is differentiable in $x_0$ with $p = Du(x_0)$ and we denote by $D^2u(x_0)$ the matrix $A$.}. Then
	\begin{equation*}\label{coherence1}
	u \in \TSH(x_0) \Leftrightarrow   D^2u(x_0)  \in \Theta(x_0).
	\end{equation*}
The forward implication makes use of the Taylor expansion for $u$ and the fact \eqref{closure_interior}, while the reverse implication uses only the positivity property $\Theta(x) + ( \{0\} \times \cP) \subset \Theta(x)$ for each $x \in \Omega$, which is contained in \eqref{PE2}. For more details see Remark 2.7 of \cite{CHLP19}. 	
\end{rem}

We will now give a useful alternate characterization of the spaces of  $\Theta$-subharmonic functions which exploits a pointwise characterization of {\em subaffine functions}. We recall that if $X \subset \R^N$ is open, $w \in \USC(X)$ is said to be {\em subaffine on $X$} and if for each domain $\Omega \subset \subset X$ and each affine function $a$ one has
\begin{equation}\label{SAX}
\mbox{	$w \leq a$ on $\partial \Omega \ \ \Rightarrow \ \ w \leq a$ on $\Omega$.} 
\end{equation} 
If \eqref{SAX} holds for each $\Omega$, we write $w \in \SA(X)$, where one knows that for $w \in \USC(X)$
\begin{equation}\label{SAC}
w \in \SA(X) \ \ \Leftrightarrow \ \ w \in \wt{\cP}(X),
\end{equation}
which means that for each $x_0 \in X$ and for each upper test function $\varphi$ for $w$ at $x_0$ one must have
\begin{equation}\label{PDSH}
D^2 \varphi(x_0) \in \wt{\cP} = \{ A \in \cS(N): \ \lambda_N(A) \geq 0\}.
\end{equation}
One also knows that $w \in \SA(X)$ if and only if for each $x_0 \in X$ there is no triple $(\veps, \rho, a)$ with $\veps, \rho > 0$ and $a$ affine such that
\begin{equation}\label{PSAC}
\mbox{$(w - a)(x_0) = 0$ \quad and \quad $(w - a)(x) \leq -\veps |x-x|^2$ for $x \in B_{\rho}(x_0)$}.
\end{equation}
We will write $w \in \SA(x_0)$ if \eqref{PSAC} holds.

The following lemma gives a pointwise characterization of the space $\QSHD(\Omega)$. Additional characterizations and properties of $\QSHD(\Omega)$ will be briefly discussed in Section \ref{sec:SAP}.

\begin{lem}\label{lem:QDSH_char} Given $w \in \USC(\Omega)$ and $x_0 \in \Omega$, one has $w \in \QSHD(x_0)$ if and only if
	\begin{equation}\label{QSHD_x0}
	w(x_0) \leq 0 \quad \text{or} \quad w \in \SA(x_0).
	\end{equation}
	\end{lem}

\begin{proof}
	Assume that $w \in \QSHD(x_0)$ but that \eqref{QSHD_x0} fails; that is, 
	\begin{equation}\label{QSHDL1}
	J^+_{x_0}w \subset \wt{\cQ}
	\end{equation}
	and
		\begin{equation}\label{QSHDL2}
	w(x_0) > 0 \quad \text{and} \quad w \not\in \SA(x_0).
	\end{equation}
	Combining \eqref{QSHDL2} with the pointwise characterization \eqref{PSAC} one must have a triple $(\veps, \rho, a)$ such that
		\begin{equation}\label{QSHDL3}
	a(x_0) = w(x_0) > 0 \quad \text{and} \quad w(x) \leq a(x) - \veps|x-x_0|^2 \  \text{for} \ x \in B_{\rho}(x_0).
	\end{equation}
	The function $\varphi(\cdot):= a(\cdot) - \veps |\cdot - x_0|^2$ is then a $C^2$ upper test function for $w$ at $x_0$ and hence by \eqref{QSHDL1} one must have
	\begin{equation*}\label{QSHDL4}
	J_{x_0} \varphi = (a(x_0), -2 \veps I) \in \wt{\cQ} = \{ (r,A) \in \R \times \cS(N): r \leq 0 \ \text{or} \  A \in \wt{\cP} \},
	\end{equation*}
	but $r = a(x_0) > 0$ by \eqref{QSHDL3} and $\lambda_N(-2 \veps I) < 0$ so $A = -2 \veps I \not \in \wt{\cP}$, a contradiction.
	
	On the other hand, if \eqref{PSAC} holds, but $w \in \QSHD(x_0)$ fails, there must be an upper test function $\varphi$ for $w$ at $x_0$ which is $C^2$ near $x_0$ and satisfies
		\begin{equation}\label{QSHDL5}
	 (w - \varphi)(x_0) = 0 \quad \text{and} \quad (w  - \varphi)(x) \leq 0 \ \text{near} \ x_0
	\end{equation}
	with 
		\begin{equation}\label{QSHDL6}
	J_{x_0} = (\varphi(x_0), D^2 \varphi(x_0)) \not\in \wt{\cQ} = \{ (r,A) \in \R \times \cS(N): r \leq 0 \ \text{or} \  A \in \wt{\cP} \}.
	\end{equation}
	One has $\varphi(x_0) = w(x_0) \leq 0$ by \eqref{QSHDL5} and \eqref{PSAC} and hence \eqref{QSHDL6} yields $A = D^2 \varphi(x_0) \not\in \wt{\cP}$. However, since $w \in \SA(x_0)$ and $\varphi$ is a $C^2$ upper test function for $w$ at $x_0$, by \eqref{SAC} one must have $A = D^2 \varphi(x_0) \in \wt{\cP}$, a contradiction. 
\end{proof}

We are now ready for the alternate characterization. 

\begin{thm}\label{thm:TSH_char}
	Let $\Theta$ be a proper elliptic map on $\Omega$.
\begin{itemize}
	\item[(a)] A function $u \in \USC(\Omega)$ is $\Theta$-subharmonic in $x_0 \in \Omega$ if and only if
\begin{equation}\label{TSubChar}
\mbox{$u + v \in \QSHD(x_0)$ for every $v$ which is $C^2$ near $x_0$ with $J_{x_0}v \in \wt{\Theta}(x_0)$.}
\end{equation}
	\item[(b)] A function $u \in \LSC(\Omega)$ is $\Theta$-superharmonic in $x_0 \in \Omega$ if and only if
\begin{equation}\label{TSuperChar}
\mbox{$-u + v \in \QSHD(x_0)$ for every $v$ which is $C^2$ near $x_0$ with $J_{x_0}v \in \Theta(x_0)$.}
	\end{equation}
\end{itemize}
	\end{thm}

\begin{proof}
	Since (b) is equivalent to the statement that $-u \in \TSHD(x_0)$, claim (b) follows from claim (a) by duality since the dual of $\wt{\cQ}(x_0)$ is $\cQ(x_0)$. We argue by contradiction.
	
	Assume first that $u \in \TSH(x_0)$ but that \eqref{TSubChar} fails; that is, one has
	\begin{equation}\label{TSH1}
	J^+_{x_0}u \subset \Theta(x_0),
	\end{equation}
	but there exists $v$ which is $C^2$ near $x_0$ satisfying
	\begin{equation}\label{TSH2}
	J_{x_0}v \in \wt{\Theta}(x_0), \ \  (u + v)(x_0) > 0 \ \ \text{and} \ \ u + v \not\in \SA(x_0),
	\end{equation}
	where the last condition in \eqref{TSH2} means that there is a triple $(\veps,\rho, a)$ with
	\begin{equation}\label{TSH3}
	(u + v - a)(x_0) = 0, \ \ \text{and} \ \ (u + v - a)(x) \leq -\veps|x-x_0|^2 \ \ \text{for} \ x \in B_{\rho}(x_0). 
	\end{equation}
	For each $\veps > 0$, consider the function $v_{\veps}$ defined by
	\begin{equation}\label{TSH4}
	v_{\veps}:= v + \veps Q_{x_0} - \veps \ \ \text{with} \ \ Q_{x_0}(x):= \frac{1}{2} |x - x_0|^2.
	\end{equation}
	Since $J_{x_0}v \in \wt{\Theta}(x_0)$ by \eqref{TSH2} one has for every $\veps > 0$
	\begin{equation*}\label{TSH5}
	J_{x_0} v_{\veps}:= J_{x_0}v + (-\veps, \veps I) \in [\wt{\Theta}(x_0)]^{\circ} = -[\Theta(x_0)]^c
	\end{equation*}
	while
		\begin{equation*}\label{TSH6}
	(u + v_{\veps})(x_0) = (u + v)(x_0) - \veps > 0  \ \ \text{if} \ \veps < (u + v)(x_0)
	\end{equation*}
	and using \eqref{TSH3} with the affine function $a_{\veps}:= a - \veps$ one has
		\begin{equation*}\label{TSH7}
	(u + v_{\veps} - a_{\veps})(x_0) = 0 \quad \text{and} \quad 	(u + v_{\veps} - a_{\veps})(x) \leq - \frac{\veps}{2}|x -x_0|^2, \ \ \forall \ x \in B_{\rho}(x_0).
	\end{equation*}
	Hence with $0 < \veps < (u + v)(x_0)$, the function $\varphi_{\veps}:= -v_{\veps} + a_{\veps}$ is an upper test function for $u$ at $x_0$ and satisfies
	 \begin{equation}\label{TSH8}
	 J_{x_0} \varphi_{\veps}:= J_{x_0}v + (a(x_0), - \veps I) \in \Theta(x_0)
	 \end{equation}
	 by \eqref{TSH1}. However, since $a(x_0) > 0$ by \eqref{TSH2} and \eqref{TSH3}, then \eqref{TSH8} yields
	 \begin{equation}\label{TSH9}
	 - J_{x_0} v = J_{x_0}\varphi_{\veps}  + (- a(x_0), \veps I) \in [\Theta(x_0)]^{\circ},
	 \end{equation}
	 which means $J_{x_0}v \in -[\Theta(x_0)]^{\circ} = \wt{\Theta}(x_0)^c$ which contradicts the first condition in \eqref{TSH2}.
	 
	 On the other hand, if \eqref{TSubChar} holds, but $u \not\in \TSH(x_0)$, then there exists $\varphi$ which is $C^2$ near $x_0$ and satisfies for some $\rho > 0$
	\begin{equation}\label{TSH10}
	(u - \varphi)(x) \leq (u - \varphi)(x_0) = 0 \ \ \text{for each} \ x \in B_{\rho}(x_0).
	\end{equation} 
 and
	 \begin{equation}\label{TSH11} 
	 J_{x_0} \varphi \not\in \Theta(x_0),
	 \end{equation}
For each $\veps > 0$ consider the function $v_{\veps}:= - \varphi + \veps - \veps Q_{x_0}$ with $Q_{x_0}$ as in \eqref{TSH4}. These $\varphi_{\veps}$ are $C^2$ near $x_0$ and satisfy
\begin{equation}\label{TSH12}
	J_{x_0} v_{\veps} = - J_{x_0} \varphi + (\veps, -\veps I),
\end{equation}
where, by \eqref{TSH11}, one has $- J_{x_0} \varphi \in - [\Theta(x_0)]^c = [\wt{\Theta}(x_0)]^{\circ}$ and hence
 \begin{equation}\label{TSH13} 
	\mbox{$J_{x_0} v_{\veps} \in \wt{\Theta}(x_0)$ for each suffiently small $\veps > 0$.}
\end{equation}
However, using \eqref{TSH10} and the definitions of $v_{\veps}$ and $Q_{x_0}$, one has
	 \begin{equation}\label{TSH14} 
(u + v_{\veps})(x_0) = (u - \varphi)(x_0) + \veps > 0, 
\end{equation} 
	 \begin{equation}\label{TSH15} 
(u + v_{\veps} - \veps)(x_0) = (u - \varphi)(x_0) = 0 
\end{equation}
and
 \begin{equation}\label{TSH16} 
(u + v_{\veps} - \veps )(x) = (u - \varphi - \veps Q_0)(x) \leq -\frac{\veps}{2} |x - x_0|^2, \ \ \forall \ x \in B_{\rho}(x_0).
\end{equation}
The formulas \eqref{TSH15}-\eqref{TSH16} with the triple $(\veps/2, \rho, \veps)$ show that $(u + v_{\veps}) \not\in \SA(x_0)$, which combined with \eqref{TSH14} says that  $J_{x_0}v \not\in \wt{\Theta}(x_0)$ and hence \eqref{TSH13} contradicts \eqref{TSubChar} for the function $v_{\veps}$ with $\veps > 0$ and small.
\end{proof}

We will often use this characterization to show that that $u \in \USC(\Omega)$ belongs to $\TSH(\Omega)$ by using an argument by contradiction. We formalize this in the following remark. 

\begin{rem}\label{rem:use_TSH} Given $u \in \USC(\Omega)$ and $x_0 \in \Omega$. If $u \not\in \TSH(x_0)$ then there exists a function $v$ which is $C^2$ near $x_0$ with $J_{x_0}v  \in \widetilde{\Theta}(x_0)$ and there exists a triple $(\veps, \rho, a)$ such that
	\begin{equation}\label{useTSH1}
	(u + v)(x_0) > 0 \quad \text{and} \quad \left\{ \begin{array}{l} (u + v - a)(x_0) = 0 \\ (u + v -a)(x) \leq -\veps |x-x_0|^2, \ \forall \ x \in B_{\rho}(x_0) \end{array} \right. ,
	\end{equation}
since \eqref{useTSH1} is the meaning of $u + v \not\in \QSHD(x_0)$ (see Lemma \ref{lem:QDSH_char} and \eqref{PSAC}). Moreover, by reducing $\veps$ and altering the affine function $a$ if need be, we can assume that $v$ satisfies the stronger condition
	\begin{equation}\label{interior_test}
	J_{x_0}v \in [\widetilde{\Theta}(x_0)]^{\circ} = -[\Theta(x_0)^c].
	\end{equation}
by considering the perturbation $v_{\veps}: = v + \veps Q_{x_0} - \veps$ used in the proof of Theorem \ref{thm:TSH_char}.
\end{rem}

\section{Continuity of proper elliptic maps}\label{sec:continuity}

In preparation for the comparison principle for proper elliptic maps and their applications to comparison principles for admissible viscosity solutions for proper elliptic branches of second order gradient free fully nonlinear equations \eqref{FNLNG}, we present a few elementary properties of $\Theta$-subharmonic functions associated to proper elliptic maps. These properties will be needed in the proof of the subharmonic addition theorem \eqref{AAT:intro} (which is stated in Theorem \ref{thm:AAT}) and these properties depend on various degrees of continuity of the proper elliptic maps $\Theta$.

We begin with describing the notion continuity that we will require, that of {\em Hausdorff continuity}. Given $\Phi \subset \R \times \cS(N)$ and $\veps > 0$ we will denote by
$$
N_{\veps} \Phi = \{ (s,B) \in \R \times \cS(N): ||(s,B) - (r,A)|| < \veps \ \text{for some} \ (r,A) \in \Phi \} = \bigcup_{(r,A) \in \Phi} B_{\veps}(r,A),
$$
the {\em $\veps$-enlargement} of the subset $\Phi$ where $\displaystyle{ ||(r,A)||:= \max \left\{ |r|, \max_{1 \leq i \leq N} |\lambda_i(A)| \right\} }$ gives a norm on $\R \times \cS(N)$. Proper elliptic maps take values in $\mathcal{E} \subset \mathcal{K}(\R \times \cS(N))$, where $\mathcal{K}(\R \times \cS(N))$ are the closed subsets of $\R \times \cS(N)$. One knows that $(\mathcal{K}(\R \times \cS(N)), d_{\mathcal{H}})$ is a complete metric space with respect to the {\em Hausdorff distance}  defined by
\begin{equation}\label{HD}
\mbox{ $d_{\mathcal{H}}(\Phi, \Psi) := \inf \{ \veps > 0: \ \Phi \subset N_{\veps}(\Psi) \ \text{and} \ \Psi \subset N_{\veps}(\Phi) \}$.}
\end{equation}
See Proposition 7.3.3 and Proposition 7.3.7 of Burago, Burago and Ivanov \cite{BBI01} for details on this structure, where we note that since the subsets of $\R \times\cS(N)$ need not be bounded, the metric can take on the value $+ \infty$; in particular, one has
\begin{equation}\label{HDempty}
\mbox{ $d_{\mathcal{H}}(\Phi, \emptyset) = +\infty$ for each non empty $\Phi \in \mathcal{K}(\R \times \cS(N))$}
\end{equation} 
and with $\cJ := \R \times \cS(N)$ one has
\begin{equation}\label{HDall}
\mbox{ $d_{\mathcal{H}}(\Phi, \cJ) = +\infty$ for each closed $\Phi \subsetneq \cJ$.}
\end{equation}

\begin{defn}\label{defn:HC} An arbitrary map $\Theta: \Omega \to \mathcal{K}(\R \times \cS(N))$ will be called {\em Hausdorff continuous} on $\Omega$ if for each $x \in \Omega$ and $\eta > 0$ there exists $\delta = \delta(x, \eta) > 0$ such that
	\begin{equation}\label{HC}
	\mbox{ $d_{\mathcal{H}}(\Theta(x), \Theta(y)) < \eta$ for each $y \in \Omega$ such that $|x-y|< \delta$.}
	\end{equation}
\end{defn}

The following (elementary) remark will have important consequences for the continuity proper elliptic maps. 

\begin{rem}\label{rem:LUC} An arbitrary map $\Theta: \Omega \to \mathcal{K}(\R \times \cS(N))$ is Hausdorff continuous on $\Omega$ if and only if $\Omega$ is {\em locally uniformly Hausdorff continuous}; that is, for any $\Omega' \subset \subset \Omega$:
	\begin{equation}\label{UHC}
	\begin{array}{c}
	\text{ for each $\eta > 0$ there exists $\delta = \delta(\eta, \Omega') > 0$ such that} \\ 
	\text{$d_{\mathcal{H}}(\Theta(x), \Theta(y)) < \eta$ for each $x,y \in \Omega'$ such that $|x-y|< \delta$.}
	\end{array}
	\end{equation}
Indeed, if $\Theta$ is continuous on $\Omega \subset \R^N$ taking values in the metric space $(\mathcal{K}(\R \times \cS(N)), d_{\mathcal{H}})$, the Heine-Cantor theorem gives the uniform continuity of $\Theta$ on $\overline {\Omega '}$ compact (and hence on $\Omega '$). On the other hand, with $x\in \Omega$ arbitrary, it is enough to consider $\Omega' = B_{\rho(x)}(x) \subset \subset \Omega$ to find the continuity of $\Theta$ at $x$. 
\end{rem}

\begin{rem}\label{rem:C=HC} 

From here on, we will use the shorter term {\em continuous} in place of {\em Hausdorff continuous} for maps $\Theta: \Omega \to \mathcal{K}(\R \times \cS(N))$.
	\end{rem}

 For proper elliptic maps, (local) uniform continuity has useful equivalent formulations. 
  
\begin{prop}\label{UCHD} Let $\Theta$ be a proper elliptic map on $\Omega$ and let $\Omega' \subseteq \Omega$. Then the following are equivalent:
	\begin{itemize}
		\item[(a)] $\Theta$ is uniformly continuous on $\Omega'$, that is \eqref{UHC} holds;
		\item[(b)] For for each $\eta > 0$ there exists $\delta = \delta(\eta, \Omega') > 0$ such that
		\begin{equation}\label{uusc1}
		\mbox{$\Theta(B_{\delta}(x)) \subset N_{\eta}(\Theta(x))$ for each $ x \in \Omega'$;}
		\end{equation}
		\item[(c)] For for each $\eta > 0$ there exists $\delta = \delta(\eta, \Omega') > 0$ such that for each $x,y \in \Omega'$
		\begin{equation}\label{uusc2}
		\mbox{$|x -y| < \delta \ \Rightarrow \ \Theta(x) + (- \eta, \eta I) \subset \Theta(y)$.}
		\end{equation}
	\end{itemize}
\end{prop}
The proof follows easily from the definitions of proper ellipticity and continuity for maps  $\Theta: \Omega \to \mathcal{K}(\R \times \cS(N))$. See the proof of Proposition 3.3 in \cite{CP17} for the analogous result for {\em uniformly continuous elliptic maps} $\Theta$, which take values in $\mathcal{K}(\cS(N))$. Notice by interchanging the roles of $x$ and $y$ in \eqref{uusc2}, one also must have $\Theta(y) + (- \eta, \eta I) \subset \Theta(x)$ if $|x-y| < \delta$.

\begin{rem}\label{rem:continuity_types}
	Property (b) of Proposition \ref{UCHD} is precisely the notion that the set valued map $\Theta :  \Omega' \to  \wp(\R \times \cS(N))$ is {\em uniformly upper semicontinuous on $\Omega'$} (see Chapter 1 of Aubin and Cellina \cite{AC84} for the elementary notions concerning set-valued maps, including their semi-continuity). Hence, Proposition \ref{UCHD} says that for {\bf proper elliptic maps} the (local) uniform upper semicontinuity of $\Theta$ as a set value map is equivalent to the (local) uniform continuity of the function $\Theta$ taking values in the metric space $\left( \mathcal{K}(\R \times \cS(N)), d_{\mathcal{H}} \right)$. Property (c) in terms of translations by multiples of $(-1, I)$ is the form in which we will normally use the (local) uniform continuity.
\end{rem}

We now show that continuity of $\Theta$ passes to the dual map and that uniform continuity extends to the boundary.

\begin{prop}\label{prop:dual_extension} Let $\Theta: \Omega \to \mathcal{E} \subset \mathcal{K}( \R \times \cS(N))$ be an elliptic map. Then,
	\begin{itemize}
		\item[(a)] $\Theta$ is (uniformly) continuous on $\Omega$ if and only the dual map $\widetilde{\Theta}$ is (uniformly)  continuous on $\Omega$. 
	\item[(b)] If $\Theta$ is uniformly  continuous on $\Omega$, then $\Theta$ extends to a uniformly continuous elliptic map on $\overline{\Omega}$.
	\end{itemize}
\end{prop}

\begin{proof}  A straightforward adaptation of Proposition 3.5 of \cite{CP17}, which can be proven using the formulation \eqref{uusc2}, shows that $\Theta$ is uniformly  continuous on $\Omega$ if and only the dual map $\widetilde{\Theta}$ is. It is worth noting that the $\delta, \eta$ relation is the same for $\Theta$ and its dual. Finally, a map that is merely  continuous on $\Omega$ is uniformly continuous on $\overline{\Omega'}$, for any $\Omega' \subset \subset \Omega$, thus its dual $\widetilde{\Theta}$ is continuous on $\Omega$, which completes part (a).

Part (b) can be proven using the argument in \cite[Proposition 3.10]{CP17}; that is, for any $x_0 \in \partial \Omega$, define $\Theta(x_0)$ in the usual way as the limiting set of the Cauchy sequence $\{\Theta(x_k)\}$, where $\{x_k\} \subset \Omega$ is an arbitrary sequence converging to $x_0$. Then, one can verify that properties \eqref{PE1} and \eqref{PE2} pass to the limit. 
\end{proof}

We conclude this section with the elementary properties of $\Theta$-subharmonic functions associated to continuous proper elliptic maps.

\begin{prop}\label{prop:TSH_properties} Let $\Theta$ be a continuous elliptic map on $\Omega$.
	\begin{enumerate}
		\item {\rm (Maximum Property)} $u,v \in \TSH(\Omega) \ \Rightarrow \ \max\{ u, v \} \in \TSH(\Omega)$;
		\item {\rm (Sliding Property)} $u \in \TSH(\Omega) \ \Rightarrow \ u - m \in \TSH(\Omega)$ for each constant $m \geq 0$;
		\item {\rm (Families Locally Bounded Above Property)} Let $\mathfrak{F} \subset \TSH(\Omega)$ be a non empty family of functions which are locally uniformly bounded from above. Then the upper envelope $\displaystyle{u:= \sup_{f \in \mathfrak{F}} f}$ has upper semicontinuous regularization \footnote{ We recall that $\displaystyle{u^*(x):= \limsup_{r \to 0^+} \{ u(y): y \in \Omega \cap \overline{B}_r(x)\}}$ for each $x \in \Omega$.} $u^* \in \TSH(\Omega)$.
	\end{enumerate}
	If, in addition, $\Theta$ is uniformly continuous on $\Omega$,
	\begin{itemize}
		\item[(4)] {\rm (Uniform Translation Property)} All sufficiently small translates of $u \in \TSH(\Omega)$ have a fixed small quadratic perturbation which is $\Theta$-subharmonic on the domain of the translate. In particular, for each $\eta > 0$ if $\delta = \delta(\eta) > 0$ is chosen as in the formulation \eqref{uusc2} of uniform continuity then \footnote{ If $\Theta$ is a constant elliptic map, then a stronger consequence than (4) follows, namely $u_y= u(\cdot - y) \in \TSH(\Omega_{\delta})$ for all $y \in B_{\delta}(0)$. This property plays a key role in \cite{HL09} but may fail if $\Theta$ is not constant.}.
		\begin{equation}\label{UTF1}
		\mbox{  $u_{y; \eta} := u( \cdot + y) + \frac{\eta}{2}  ( |\cdot |^2 - \omega ) \in \TSH(\Omega_{\delta}), \ \forall \ y \in B_{\delta}(0)$,}
		\end{equation}
		with
		\begin{equation}\label{UTF2}
		\Omega_{\delta}:= \{ x \in \Omega: \ {\rm dist}(x, \partial \Omega) > \delta \} \quad \text{and} \quad \omega:= 2 + \sup_{x \in \Omega} |x|^2
		\end{equation} 
		\item[(5)] {\rm (Existence of Bounded $\Theta$-harmonics)}  There exist smooth bounded $\Theta$-harmonic functions on $\Omega$ of the form
		$$
		\varphi(\cdot):= - \tau + \frac{\tau}{2}|\cdot|^2
		$$
		for each sufficiently large $\tau$.
	\end{itemize}
Moreover, one also has the properties (1) - (5) for $\wt{\Theta}$-subharmonic functions since $\wt{\Theta}$ is (uniformly)  continuous if $\Theta$ is by Proposition \ref{prop:dual_extension}
\end{prop}

\begin{proof}
	We will make use of the characterization formula \eqref{TSubChar} and argue by contradiction (as  discussed in Remark \ref{rem:use_TSH}) to prove the Maximum Property (1) and the Sliding Property (2). If (1) were false, then $w:=\max\{ u, v \} \not\in \TSH(x_0)$ for some $x_0 \in \Omega$ and there exist $\varphi$ which is $C^2$ near $x_0$ with  $J_{x_0}\varphi \in \widetilde{\Theta}(x_0)$ and a triple $(\veps, \rho, a)$ such that
	\begin{equation}\label{M1}
	(w + \varphi)(x_0) > 0 \quad \text{and} \quad \left\{ \begin{array}{l} (w + \varphi - a)(x_0) = 0 \\ (w + \varphi -a)(x) \leq -\veps |x-x_0|^2, \ \forall \ x \in B_{\rho}(x_0) \end{array} \right. .
	\end{equation}
	Since $u,v \leq w$ everywhere, in \eqref{M1} we can replace $w$ with $u$ when $w(x_0) = u(x_0)$ or $w$ with $v$ when $w(x_0) = v(x_0)$ to contradict $u,v \in \TSH(x_0)$.
	
	Similarly, if (2) were false, then for some $x_0 \in \Omega$ there exist $\varphi$ which is $ C^2$ near $x_0$ with  $J_{x_0}\varphi \in \widetilde{\Theta}(x_0)$ and a triple $(\veps, \rho, a)$ such that
	\begin{equation}\label{M2}
	(u - m + \varphi)(x_0) > 0 \quad \text{and} \quad \left\{ \begin{array}{l} (u - m + \varphi - a)(x_0) = 0 \\ (u - m + \varphi -a)(x) \leq -\veps |x-x_0|^2, \ \forall \ x \in B_{\rho}(x_0) \end{array} \right. .
	\end{equation}
	Setting  $\tilde{\varphi} := \varphi - m$ in \eqref{M2} gives a contradiction to $u \in \TSH(x_0)$ since $J_{x_0}\tilde{\varphi}  = J_{x_0} \varphi + (-m, 0) \in \widetilde{\Theta}(x_0)$.
	
	For the Families locally Bounded Above Property (3), it suffices to show that for each fixed $x_0 \in \Omega$ and each fixed $v$ which is $C^2$ near $x_0$ with $J_{x_0}v \in \wt{\Theta}(x_0)$ one has
	\begin{equation}\label{P1}
	w^*:= u^* + v \in \QSHD(x_0),
	\end{equation}
	where $w^*(x) = (u+v)^*(x)$ for all $x \in \Omega$ by the continuity of $v$. Using the local uniform continuity of $\wt{\Theta}$ with a sequence $\{ \veps_j \}_{j \in \N}$ such that $\veps_j \searrow 0$ as $j \to +\infty$ one has the existence of $\delta_j = \delta_j(\veps_j/2)$ for which 
	\begin{equation}\label{P2}
	J_{x_0}v + \left( - \frac{\veps_j}{2}, \frac{\veps_j}{2} I \right) \in \wt{\Theta}(x) \ \ \text{for each} \ x \in B_{\delta_j}(x_0).
	\end{equation}
	With $Q_{x_0}(\cdot) := \frac{1}{2} | \cdot - x_0|^2$ consider the sequence of functions
	\begin{equation*}\label{P3}
	w_j:= \sup_{f \in \mathfrak{F}} \left(f + v + \veps_j Q_{x_0} - \frac{\veps_j}{2} \right) = u + v + \veps_j Q_{x_0}\ \ j \in \N,
	\end{equation*}
	where $f + v + \veps_j Q_{x_0} - \frac{\veps_j}{2} \in  \QSHD(B_{\delta_j}(x_0))$ for each $j \in \N$ by \eqref{TSubChar} since $J_{x_0} \left( v + \veps_j Q_{x_0} - \frac{\veps_j}{2} \right) \in \wt{\Theta}(x_0)$ follows from \eqref{P2}. Property (3) for the constant coefficient gradient free proper elliptic map $\cQ$ on the open set $B_{\delta_j}(x_0)$ yields (see Proposition B.1 (F) of \cite{CHLP19}):
	\begin{equation}\label{P4}
	w^* + \veps_j Q_{x_0}  - \frac{\veps_j}{2} = u^* + v + \veps_j Q_{x_0}  - \frac{\veps_j}{2} = w_j^* \in  \QSHD(B_{\delta_j}(x_0)), \ \ j \in \N.
	\end{equation}
	If \eqref{P1} were false, then there exists a triple $(\veps, \rho, a)$ such that
	\begin{equation}\label{P5}
	w^*(x_0) > 0, \ \  (w^* - a)(x_0) = 0 \ \ \text{and} \  (w^* - a)(x) \leq - \veps |x - x_0|^2, \ \ \forall \ x \in B_{\rho}(x_0).
	\end{equation}
	By taking $j^*$ large enough to ensure $\veps_{j^*} \leq \veps$ and $\veps_{j^*}/2 < w^*(x_0)$ one has
	\begin{equation}\label{P6}
		 w_{j^*}^*(x_0) = w^*(x_0) - \frac{\veps_{j^*}}{2}  > 0 
	\end{equation}
	since $\veps_{j^*}/2 < w^*(x_0)$  and with the affine function $\wt{a}(\cdot) := a(\cdot) - \veps_{j^*}/2$ one has
	\begin{equation}\label{P7}
	(w_{j^*}^* - \wt{a})(x_0) = 0 \quad \text{and} \quad (w_{j^*}^*  - \wt{a})(x) \leq - \frac{\veps}{2} |x - x_0|^2, \ \ \forall \ x \in B_{\rho}(x_0).
	\end{equation}
	The relations in \eqref{P6} and \eqref{P7} say that $w^*_{j^*} \not\in \SA(x_0)$ which contradicts \eqref{P4}.
	
	For the Uniform Translation Property (4), we begin by noting that $ u_{y; \eta}$ is well defined and uniformly continuous on $\Omega_{\delta}$ for each $\eta > 0$ and $y \in B_{\delta}(0)$. It remains to show that for each $x_0 \in \Omega_{\delta}$ if $v \in C^2(\Omega)$ satisfies
	\begin{equation}\label{bl1}
	J_{x_0} v \in \widetilde{\Theta}(x_0) \quad \text{and} \quad ( u_{y; \eta} + v)(x_0) > 0
	\end{equation}
	then
	\begin{equation}\label{bl2}
	u_{y; \eta} + v \in \SA(x_0).
	\end{equation}
	Define the test function $\hat{v}_{y; \eta}$ by $\hat{v}_{y; \eta}(x):= v(x-y) + \frac{\eta}{2} ( |x-y|^2 - \omega)$ with $\omega$ as in \eqref{UTF1} - \eqref{UTF2}. Notice that
	\begin{equation}\label{bl3}
	J_{x_0 + y}\hat{v}_{y; \eta} = \left( v(x_0) +  \frac{\eta}{2} (|x_0|^2 - \omega), D^2v(x_0) + \eta I \right)
	\end{equation}
	where $\frac{\eta}{2}(|x_0|^2 - \omega) \leq - \eta$. Hence \eqref{bl3} yields
	$$
	J_{x_0 +y}\hat{v}_{y; \eta} = J_{x_0}v + (-\eta, \eta I) + \left( \frac{\eta}{2} (|x_0|^2 - \omega) + \eta , 0 \right)
	$$
	and hence for each $\eta > 0, y \in B_{\delta}(x_0)$ one has
	\begin{equation}\label{bl4}
	J_{x_0 + y}\hat{v}_{y; \eta} \in \widetilde{\Theta}(x_0 + y) + \left( \frac{\eta}{2} (|x_0|^2 - \omega) + \eta , 0 \right) \in \widetilde{\Theta}(x_0 + y)
	\end{equation}
	by the uniform continuity of $\widetilde{\Theta}$ and the non positivity of $\frac{\eta}{2}(|x_0|^2 - \omega) + \eta$. In addition one has
	\begin{equation}\label{bl5}
	u(x_0+y) + \hat{v}_{y; \eta}(x_0 + y) = u_{y; \eta}(x_0) + v(x_0) > 0
	\end{equation}
	where the positivity comes from \eqref{bl1}. Since $x_0 + y \in \Omega$, one has $u \in \TSH(x_0 + y)$ and hence \eqref{bl4} and \eqref{bl5} give $u + \hat{v}_{y; \eta} \in \SA(x_0 + y)$
	and hence
	\begin{equation}\label{bl6}
	u(\cdot + y) + \hat{v}_{y; \eta}(\cdot + y) \in \SA(x_0).
	\end{equation}
	However, using the definitions
	$$
	u(\cdot + y) + \hat{v}_{y; \eta}(\cdot + y) =  u_{y; \eta}(\cdot) + v(\cdot),
	$$
	and hence \eqref{bl6} gives the needed conclusion \eqref{bl2}. 
	
	For the existence of bounded $\Theta$-harmonic functions in (5), it suffices to show that there exists $\tau$ such that
	\begin{equation}\label{gTH1}
	J_x \varphi = ( -\tau , \tau I) \in \Theta(x) \cap \widetilde{\Theta}(x), \ \ \forall \ x \in \Omega.
	\end{equation}
	Moreover, it suffices to construct $\tau$ such that $( -\tau , \tau I) \in \Theta(x)$ holds, since a corresponding $\tilde{\tau}$ can be constructed for $\widetilde{\Theta}$ (which is also uniformly continuous by Proposition \ref{prop:dual_extension} (a)) and hence one can take the maximum of $\tau$ and $\tilde{\tau}$ by the monotonicity of proper elliptic maps.
	
	Note that $\Theta$ extends to a uniformly continuous map on $\overline{\Omega}$ by Proposition \ref{prop:dual_extension} (b), which is compact. Since $\Theta(y)$ is a proper elliptic set for each $y \in \overline{\Omega}$, there exists $t_y$ such that
	\begin{equation}\label{gTH2}
	\mbox{$(-t, tI) \in \Theta(y)$ for each $t \geq t_y$.}
	\end{equation}
	Indeed, pick any $(r_y,A_y) \in \Theta(y)$ and define $t_y:= \max \{ - r_y, \lambda_N(A_y) \}$ and one has
	$$
	(-t_y, t_y I) = (r_y, A_y) + (-r_y -t_y, t_yI - A_y) \in (r_y, A_y) + \Q \subset \Theta(y).
	$$
	Using the uniform continuity of $\Theta$ with $\eta = 1$, there exists $\delta = \delta(1) > 0$ such that
	\begin{equation}\label{gTH3}
	\mbox{$\Theta(y) + (-1, I) \in \Theta(x)$ for each $x,y \in \overline{\Omega}$ with $|x-y| < \delta$.}
	\end{equation}
	Since $\overline{\Omega}$ is compact there exists a finite open covering $\left\{ B_{\delta}(y_k) \right\}_{k=1}^n$ with $y_k \in \overline{\Omega}$. Combining \eqref{gTH2} and \eqref{gTH3}, one has for each $x \in \overline{\Omega}$
	$$
	\mbox{$(1 + t)(-1, I) \in \Theta(x)$ provided that $\displaystyle{t \geq T:= \max_{1 \leq k \leq n} t_{y_k}}$.}
	$$
	Picking $\tau = 1 + T$ yields the desired conclusion  $( -\tau , \tau I) \in \Theta(x)$ .
\end{proof}

A pair of remarks are in order concerning the properties (1)-(5) of Proposition \ref{prop:TSH_properties}.

\begin{rem}\label{rem:TSH_regularity} The properties (1) and (2) hold for arbitrary proper elliptic maps $\Theta$ since they are purely pointwise statements that require no regularity of $\Theta$. The argument used for property (3) does not use really require continuity \eqref{HC}. It would suffice to ask that there exists $\delta_j = \delta_j(\veps_j, x_0, D^2v(x_0))$ such that \eqref{P2} holds. This is because the argument in purely local near each fixed $x_0$, and fixed element of $\wt{\Theta}(x_0)$. On the other hand, uniform continuity is really used for properties (4) and (5).
\end{rem}

\begin{rem}\label{rem:TSH_use} The maximum property (1), sliding property (2) and the bounded $\Theta, \wt{\Theta}$-subharmonics in property (5) will be used to make suitable truncations in the reduction of the Subharmonic Addition Theorem for semi-continuous functions to the case in which the functions are bounded from below (see Lemma \ref{lem:AAT_reduction_BB} below). On the other hand, the families locally bounded above property (3) and the uniform translation property (4) will be used in the proof of the Subharmonic Addition Theorem in the case of functions which are bounded from below (see Lemma \ref{lem:SCA} below).
	\end{rem}

\section{A maximum principle for $\wt{\cQ}$-subharmonic functions}\label{sec:SAP}

In this section, we discuss some fundamental properties of the space of $\wt{\cQ}$-subharmonic functions which play a key role in our treatment of the comparison principle. We begin by noting that the space $\QSHD(X)$ with $X \subseteq \R^N$ open has also been studied in \cite{CHLP19} in the context of constant coefficient (and gradient free) subequation constraint sets. There one finds additional characterizations such as
\begin{equation}\label{SAPlusChar}
w \in \QSHD(X) \ \ \Leftrightarrow \ \ w \in \SA^+(X) \ \ \Leftrightarrow \ \ w^+ \in \SA(X),
\end{equation}
where $w^+$ is the positive part of $w$, $\SA(X)$ are the subaffine functions satisfying the comparison principle \eqref{SAX} for each $\Omega \subset \subset X$ and each affine function $a$. The space $\SA^+(X)$ consists of the {\em subaffine plus functions} on $X$ in which one uses positive affine functions $a$ in the comparison principle \eqref{SAX}; that is, $w \in \USC(X)$ is subaffine plus on $X$ if for each $\Omega \subset \subset X$ and each affine function $a$ which is non-negative on $\overline{\Omega}$, one has   
\begin{equation*}
\mbox{	$w \leq a$ on $\partial \Omega \ \ \Rightarrow \ \ w \leq a$ on $\Omega$.} 
\end{equation*} 
 The equivalence \eqref{SAPlusChar} is discussed and proven in \cite{CHLP19} (see  Theorem 9.7). 
 
An important property of subaffine plus functions is the validity of the following  {\em zero maximum principle}, which is a comparison principle between the subaffine plus functions $w$ and $0$. 

\begin{thm}\label{thm:ZMP}
For each $w \in \USC(\overline{\Omega}) \cap \QSHD(\Omega)$, one has
\begin{equation}\label{ZMP}
w \leq 0 \ \text{on} \ \partial \Omega  \ \ \Rightarrow \ \ w \leq 0 \ \text{on} \ \Omega,
\end{equation}
	\end{thm}

\begin{proof}
Once the equivalence \eqref{SAPlusChar} is established, it is sufficient to observe that since $w \in \USC(\overline{\Omega})$ and $w \le 0$ on $\overline{\Omega}$, by a standard compactness argument one has that for all $\veps > 0$ there exists $\delta = \delta(\veps) > 0$ such that $w \le  \veps$ in a neighborhood of $\partial \Omega$. Then, since $w$ is subaffine plus on $\Omega$, the comparison principle with the affine function $a \equiv \veps$ holds, hence $w \le \veps$ on $\Omega$. Letting $\veps \to 0$ gives the result.
\end{proof}

\begin{rem}\label{rem:CP_regular}
Combining Theorem \ref{thm:ZMP} with Theorem \ref{thm:TSH_char}(a) immediately gives the following comparison result for proper elliptic maps $\Theta$ on $\Omega$ between viscosity subharmonics and classical superharmonics: for each pair $u \in \USC(\overline{\Omega}) \cap \TSH(\Omega)$ and $v \in C^2(\Omega) \cap C(\overline{\Omega}) \cap \TSHD(\Omega)$ one has
\begin{equation}\label{Definition_Comparison}
u + v \leq 0 \ \text{on} \ \partial \Omega  \ \ \Rightarrow \ \ u + v \leq 0 \ \text{on} \ \Omega,
\end{equation}
since $w:= u + v \in \USC(\overline{\Omega}) \cap \QSHD(\Omega)$. The main result of this paper will be to show that a continuity property on $\Theta$ ensures that \eqref{Definition_Comparison} continues to hold if $v$ is merely $\USC(\overline{\Omega}) \cap \TSHD(\Omega)$.
\end{rem}

Finally, we note that since the constant proper elliptic map $\wt{\cQ}$ is trivially uniformly continuous on $\Omega$, we have the validity of all of the properties of Proposition \ref{prop:TSH_properties}. In addition, one has the following property for decreasing limits, which plays a key role in the proof of the comparison principle in the next section.

\begin{lem}\label{lem:DLP}
	 If $\displaystyle{\{w_n\}_{n \in \N} \subset \QSHD(\Omega)}$ is a decreasing sequence, then 
$$
	 w:= \lim_{n \to +\infty} w_n \in \QSHD(\Omega).
 $$
	\end{lem}

\begin{proof}
	This is a special case of Proposition B.1 (E) of \cite{CHLP19} for the constant coefficient gradient free subequation constraint set 
	$$
	\cF:= \{ (r,p,A) \in \R \times \R^N \times \cS(N): \ (r,A) \in \wt{\cQ} \}.
	$$
\end{proof}

\section{The comparison principle for continuous proper elliptic maps}\label{sec:CP}

The purpose of this section is to prove the following comparison principle in nonlinear variable coefficient gradient-free potential theory.

\begin{thm}[Comparison principle: potential theoretic version]\label{thm:CP} Let $\Theta$ be a continuous proper elliptic map on $\Omega$. Then the comparison principle holds; that is,
if $u \in \USC(\overline{\Omega})$ and $v \in \LSC(\overline{\Omega})$ are $\Theta$-subharmonic and $\Theta$-superharmonic respectively in $\Omega$, then
\begin{equation}\label{CP}
\mbox{$u \leq v$ on $\partial \Omega \ \ \Rightarrow \ \ u \leq v$ in $\Omega$.}
\end{equation}
\end{thm}

\begin{proof}
By exploiting Harvey-Lawson duality and the zero maximum principle for $\wt{\cQ}$-subharmonic functions (Theorem \ref{thm:ZMP}), the proof of Theorem \ref{thm:CP} reduces to the proof of the following result. Recall that $v$ is $\Theta$ superharmonic if and only if $\tilde{u} = -v$ is $\wt{\Theta}$-subharmonic.

\begin{thm}[Subharmonic Addition]\label{thm:AAT} Let $\Theta$ be a uniformly continuous proper elliptic map on an open set $X \subset \subset \R^N$. For each pair of functions, $ u, \wt{u} \in \USC(\overline{X})$ one has
	\begin{equation}\label{AAT}
	u \in \TSH(X) \quad \text{and} \quad \wt{u} \in \TSHD(X) \ \ \Rightarrow \ \ u + \wt{u} \in \QSHD(X)
	\end{equation}
\end{thm}

Indeed, let $u$ and $v$ be as in the statement of Theorem \ref{thm:CP} and assume the validity of Theorem \ref{thm:AAT}. Set $\wt{u} := - v \in \USC(\overline{\Omega}) \cap \TSHD(\Omega)$. The comparison principle \eqref{CP} is then equivalent to 
\begin{equation}\label{alt_CP}
\mbox{$u + \wt{u} \leq 0$ on $\partial \Omega \ \ \Rightarrow \ \ u + \wt{u} \leq 0$ in $\Omega$.}
\end{equation}
Now set $w:= u + \wt{u}$, and consider an arbitrary open set $X \subset \subset \Omega$. Since $\Theta$ is uniformly continuous on $\overline X$, one has that $w \in \QSHD(X)$ by Theorem \ref{thm:AAT}. Since $X \subset \subset \Omega$ is arbitrary, $w \in \QSHD(\Omega)$, and \eqref{alt_CP} is a consequence of the zero maximum principle Theorem \ref{thm:ZMP} for $w$. This completes the proof of Theorem \ref{thm:CP}, modulo the proof of Theorem \ref{thm:AAT}.
\end{proof}

\begin{proof}[Proof of Theorem \ref{thm:AAT}] The proof involves three steps:
\begin{enumerate}
	\item prove \eqref{AAT} under the additional assumption that $u, \wt{u}$ are {\em semi-convex}
	and hence almost everywhere twice differentiable (Lemma \ref{lem:AAT_SC} and Lemma \ref{lem:Slod});
	\item reduce the general case to the case of $u, \wt{u}$ {\em bounded from below} by suitable truncations and limit procedures (Lemma \ref{lem:AAT_reduction_BB});
	\item prove \eqref{AAT} for $u, \wt{u}$ semi-continuous and bounded below by taking  decreasing limits of suitable quadratic perturbations of {\em sup-convolution approximations} which are semi-convex and locally subharmonic (Lemma \ref{lem:SCA}). 
\end{enumerate}

\vspace{1ex}
\noindent {\bf Step 1:} {\em Prove \eqref{AAT} in the special case 
of $u, \wt{u}$ semi-convex on $X \subset \subset \R^N$ for an arbitrary proper elliptic map (not necessarily continuous).}

\vspace{1ex}
Recall that if $\lambda > 0$,  a function $u : X \to \R$ is said to be {\em $\lambda$-semi-convex} if $u + \lambda Q_0$ is a convex function, where $Q_0(x) = \frac{1}{2}|x|^2$. 

\begin{lem}\label{lem:AAT_SC}
	 Let $\Theta$ be a proper elliptic map on $X \subset \subset \R^N$. If $u, \wt{u} \in \USC(\overline{X})$ are $\lambda$-semi-convex, then \eqref{AAT} holds; that is, 
	 \begin{equation}\label{AAT1}
	 	u \in \TSH(X) \ \text{and} \ \wt{u} \in \TSHD(X) \ \ \Rightarrow \ \ u + \wt{u} \in \QSHD(X).
	 \end{equation}
	 Moreover, one has the comparison principle \eqref{CP} on $X$ for $u$ and $v := - \wt{u} \in \LSC(\overline{X})$ which is $\Theta$-superharmonic in $X$.
	\end{lem}

\begin{proof}
	The functions $u,\wt{u}$ and $u + \wt{u}$ are all semi-convex and hence twice differentiable almost everywhere in $X$ by Alexandroff's theorem. Since $u \in \TSH(X)$ and $\wt{u} \in \TSHD(X)$, the coherence property of Remark \ref{coherence} yields
	$$
	\mbox{$J_{x_0}u \in \Theta(x_0)$ and $J_{x_0}\wt{u} \in \widetilde{\Theta}(x_0)$ for almost every $x_0 \in X$.}
	$$
	Property \eqref{sum_of_duals} of Proposition \ref{Elem_Props} then gives \footnote{ The formula $\Theta(x_0) + \wt{\Theta}(x_0) \subset \wt{\cQ}$ for each $x_0$ is known as the {\em jet addition theorem} which follows from Harvey-Lawson duality and the invariance property $\Theta(x_0) + \cQ \subset \Theta(x_0)$ (see Section 6 of \cite{CHLP19}).}
	\begin{equation}\label{AAT2}
	\mbox{$J_{x_0}(u+\wt{u}) \in \widetilde{\Q}$ for almost every $x_0 \in X$.}
	\end{equation}
	The desired conclusion \eqref{AAT1} is reached by applying the following lemma to $w:=u +\wt{u}$, which is a version of Jensen's lemma on the passage of almost everywhere to everywhere information \footnote{ See \cite{HL13} for a discussion on the equivalence of the Slodkowski and Jensen lemmas.}.
\end{proof}

\begin{lem}\label{lem:Slod} Let $w \in \USC(X)$ be $2\lambda$-semi-convex. Then $w \in \QSHD(X)$ provided that
\begin{equation}\label{Slod1}
\mbox{$J_{x}w \in \widetilde{\Q}$ for almost every $x \in X$.}
\end{equation}
\end{lem}
\begin{proof} By the pointwise characterization of Lemma \ref{lem:QDSH_char}, one needs only to show that
\begin{equation}\label{Slod2}
\mbox{ $w(x_0) \leq 0$ \ or \ $w \in \SA(x_0)$ \ for each \ $x_0 \in X$.}
\end{equation}
Define $X^+ := \{ x \in X: \ w(x) > 0 \}$, which is open since $w$ is continuous, and it suffices to show that
\begin{equation}\label{Slod3}
\mbox{$w \in \SA(x_0)$ for every $x_0 \in X^+$.}
\end{equation}
Since $w$ is twice differentiable almost everywhere, the hypothesis \eqref{Slod1} yields
$$
\mbox{$w(x) \leq 0$ \ or \ $D^2w(x) \in \widetilde{\cP}$ \ for almost every \ $x \in X$.}
$$
and hence one has
\begin{equation}\label{Slod4}
\mbox{$D^2w(x) \in \widetilde{\cP}$ \ for almost every \ $x \in X^+$.}
\end{equation}
For $w \in \USC(X^+)$ and semi-convex, the condition \eqref{Slod4} gives the needed property \eqref{Slod3} by applying Lemma 7.3 of \cite{HL09} (see also Lemma 4.10 of \cite{CP17}). We note only that the main idea is to pass a lower bound on the largest eigenvalue from a set of full measure to the entire domain $X^+$, where the condition $D^2w(x) \in \widetilde{\cP}$ means precisely $\lambda_N(D^2w(x)) \geq 0$ and the tool used is Slodkowski's largest eigenvalue theorem \cite{Sl84}.
\end{proof}

\vspace{1ex}
\noindent {\bf Step 2:} {\em For upper semi-continuous $u, \wt{u}$ and   $\Theta$ uniformly continuous on $X$, reduce to the special case of  $u, \wt{u}$ bounded from below.}

\begin{lem}\label{lem:AAT_reduction_BB} Let $\Theta$ be a uniformly continuous proper elliptic map on $X$. If the subharmonic addition theorem
		 \begin{equation}\label{AAT3}
	u \in \TSH(X) \ \text{and} \ \wt{u} \in \TSHD(X) \ \ \Rightarrow \ \ u + \wt{u} \in \QSHD(X).
	\end{equation}
holds for each pair $u, \wt{u} \in \USC(\overline{X})$ which are bounded from below, then \eqref{AAT3} holds for each pair $u, \wt{u} \in \USC(\overline{X})$.
\end{lem}
\begin{proof}
	If either $u$ or $\wt{u}$ is not bounded from below on $\overline{X}$, consider the sequences in $\USC(\overline{X})$
	\begin{equation}\label{CPr1}
	\mbox{$u_m:= \max \{u, \varphi - m\}$ \quad and  \quad $\wt{u}_m:= \max \{ \wt{u}, - \varphi - m\}$ \quad for each $m \in \N$,}
	\end{equation}
	where $\varphi$ is the bounded $\Theta$-harmonic function constructed in Proposition \ref{prop:TSH_properties} (5). These sequences will be bounded from below since $\varphi$ and $-\varphi$ are. By parts (1) and (2) of Proposition \ref{prop:TSH_properties}, one has $u_m \in \TSH(X)$ and $\wt{u}_m \in \TSHD(X)$ for each $m \in \N$. Assuming the \eqref{AAT3} holds for pairs which are bounded below, one has
	\begin{equation}\label{CPr2}
	\mbox{$ w_m := u_m + \wt{u}_m \in \QSHD(X)$ \ for each $m \in \N$, }
	\end{equation}
	but $w_m \searrow w:= u + \wt{u}$ as $m \to + \infty$ and hence $u + \wt{u} \in \QSHD(X)$ by the Decreasing Limit Property of Lemma \ref{lem:DLP} for $\QSHD(X)$.
\end{proof}

\vspace{1ex}
\noindent {\bf Step 3:} {\em Prove \eqref{AAT} for $u, \wt{u} \in \USC(\overline{X})$ which are bounded below with $\Theta$ uniformly continuous on $X$.}
\vspace{1ex}

The idea of the proof is to use the sup convolution and suitable quadratic perturbations to build regularizing sequences for $u$ and $\wt{u}$ which are semi-convex and locally $\Theta$ and $\wt{\Theta}$-subharmonic respectively. The subharmonic addition theorem holds along the approximating sequences which tend to $u + \wt{u} \in \QSHD(X)$ by the Decreasing Limit Property of Lemma \ref{lem:DLP}.

We begin by recalling that if $u \in \USC(X)$ and bounded on $X$, for each $\veps > 0$, one defines the {\em sup-convolution} $u^{\veps}$ by 
\begin{equation}\label{defn:sup_conv}
u^\veps(x) = \sup_{z \in \R^N} \left\{ u(x-z) - \frac{1}{\veps}|z|^2 \right\} \quad \forall x \in X,
\end{equation}
where one extends $u$ to be $-\infty$ outside of $X$. The function defined in \eqref{defn:sup_conv} satisfies the following well-known properties (cf.\ Theorem 8.2 of \cite{HL09}, for example):
$$
\mbox{$u^\veps$ decreases to $u$ as $\veps \rightarrow 0$}
$$
and
$$
\mbox{$u^\veps$ is $\frac{2}{\veps}$-semi-convex.}
$$
For $u \in \TSH(X)$ bounded with $|u| \leq M$ on $X$, consider the family of quadratic perturbations $u^\veps( \cdot) + \eta (| \cdot |^2 - \omega)$ with $\eta > 0$ small and $\omega := 2 + 
\sup_{x \in X}|x|^2$ is the parameter introduced in \eqref{UTF1}-\eqref{UTF2}.

\begin{lem}\label{lem:SCA} For every $\eta > 0$ there exists $\overline{\veps} = \overline{\veps}(\eta) > 0$ such that
	\begin{equation}\label{sup_conv3}
	u^\veps( \cdot) + \eta ( | \cdot |^2 - \omega) \in \TSH(X_{\delta}), \ \ \forall \ \veps \in (0, \overline{\veps}(\eta)],
	\end{equation}
	where
$$\mbox{$\delta := \sqrt{2 \veps M}$ and $X_{\delta} := \{ x \in X: \ {\rm dist}(x, \partial X) > \delta \}$}$$
\end{lem}

\begin{proof}
Indeed, the Uniform Translation Property (4) of Proposition \ref{prop:TSH_properties} says that for each $\eta > 0$ there exists $\delta = \delta(2 \eta) > 0$ such that
\begin{equation}\label{sc1}
u_{z, \eta}(\cdot) := u(\cdot - z) + \eta ( | \cdot |^2 - \omega ) \in \TSH(X_{\delta}), \ \ \forall \ z \in B_{\delta}(0).
\end{equation}
Moreover, as noted in Proposition \ref{prop:dual_extension}, the $\eta,\delta$ relation is the same for the dual map $\widetilde{\Theta}$ and hence there is an analogous family $\{ \wt{u}_{z, \eta} \}_{z \in B_{\delta}(0)}$ associated to $\wt{u}$ which will be $\wt{\Theta}$-subharmonic.

Now, for $\veps > 0$, consider the collection
\begin{equation}\label{CPb3}
\mathfrak{F} := \left\{  u( \cdot - z) - \frac{1}{\veps}|z|^2 + \eta ( | \cdot |^2 -  X ), \ \ |z| < \delta \right\}.
\end{equation}
Since $- \frac{1}{\veps}|z|^2 \leq 0$ for each $z$,  Proposition \ref{prop:TSH_properties} (2) gives $\mathfrak{F} \subset \TSH(X_{\delta})$ and the collection is locally uniformly bounded from above. By Proposition \ref{prop:TSH_properties} (3), the Perron function defined for $x \in X_{\delta}$ by 
\begin{align}\label{perron_function}
u_{\eta}^{\veps}(x)  & := \sup_{|z| < \delta} \left\{  u( x - z) - \frac{1}{\veps}|z|^2 + \eta (| x |^2 - \omega) \right\}  \nonumber\\
	& = \sup_{|z| < \delta} \left\{  u( x - z) - \frac{1}{\veps}|z|^2 \right\} + \eta (| x |^2 - \omega) 
\end{align}
will admit an upper semicontinuous regularization $[u_{\eta}^{\veps}]^*$ which belongs to $\TSH(X_{\delta})$. 

It is not hard to see that for small $\veps$ one has that $u^{\veps}_{\eta}$ is semi-convex and hence continuous so that $u^{\veps}_{\eta} = [u_{\eta}^{\veps}]^* \in \TSH(X_\delta)$ and the claim \eqref{sup_conv3} follows. Indeed, by choosing  $\veps \in (0, \delta^2(\eta)/2M)$ the values of $z$ with $|z| \geq \delta$ do not compete in the sup which defines the sup-convolution $u^{\veps}$ in \eqref{defn:sup_conv} and hence the first term in \eqref{perron_function} is $u^{\veps}(x)$ and one has the following identity on $X_{\delta(\eta)}$:
\begin{equation}\label{CPb6}
u^{\veps}_{\eta}(x) = u^{\veps}(\cdot) + \eta( |\cdot|^2 - \omega)  \ \ \text{for} \ \veps \in (0, \bar{\veps}] \ \ \text{with} \ \bar{\veps} = \frac{\delta^2(\eta)}{2M},
\end{equation}
which gives the semi-convexity and continuity of $u^{\veps}_{\eta}$. 
\end{proof}

Armed with Lemma \ref{lem:SCA}, we complete Step 3 of the proof of Theorem \ref{thm:AAT} by fixing a sequence $\{ \eta_j \}_{j \in \N}$ with $\eta_j \to 0$ as $j \to + \infty$ and select $\veps_j := \min \{\eta_j, \bar{\veps}(\eta_j) \}$ so that $\delta_j:= \delta(\veps_j) = \sqrt{2 \veps_j M} \to 0^+$ and $\Omega_{\delta_j} \nearrow \Omega$. The corresponding approximating sequences $\{u^{\veps_j}_{\eta_j} \}$ and $ \{\wt{u}^{\veps}_{\eta_j} \}$ defined by \eqref{CPb6} are $2/\veps_j$-semi-convex and $\Theta, \widetilde{\Theta}$-subharmonic in $\Omega_{\delta_j}$. By Lemma \ref{lem:AAT_SC} one has $w_j:= u^{\veps_j}_{\eta_j} + \wt{u}^{\veps}_{\eta_j} \in \QSHD(\Omega_j)$. By construction $w_j \searrow u + \wt{u}$ and $\Omega_{\delta_j} \nearrow \Omega$ and hence $u + \wt{u} \in \QSHD(\Omega)$ by applying the Decreasing Limit Property of Lemma \ref{lem:DLP}. 
\end{proof}

\subsection{Comparison with local continuity in $r$}

We conclude the section by observing that in order to prove comparison, our continuity demands on $\Theta$ can be slightly reformulated in ways that might be useful for applications. Recall that by Proposition \ref{UCHD} (c), uniform continuity of $\Theta$ on $X \subset \subset \Omega$ is equivalent to the following property: for all $\eta > 0$, $\exists \, \delta = \delta(\eta, X)$ such that
\begin{equation}\label{cont}
	\text{$x,y \in X$, $|x-y| < \delta$ $\Rightarrow$ $\Theta(x) + (- \eta, \eta I) \subset \Theta(y)$.}
\end{equation}

Assume now that $\Theta$ satisfies for each $X \subset \subset \Omega$
\begin{equation}\label{cont_cut}
	\begin{array}{c}
		\text{for all $R > 0$ large and $\eta > 0$, $\exists \, \delta = \delta(R, \eta, X)$ such that} \\
		\text{$x,y \in X$, $|x-y| < \delta$ $\Rightarrow$ $\Theta(x) \cap \left( [-R,R] \times \cS(N) \right) + (0, \eta I) \subset \Theta(y)$.}
	\end{array}
\end{equation}

\begin{thm}\label{comparison_cut} Let $\Theta$ be a proper elliptic map on $\Omega$ that satisfies \eqref{cont_cut} for each $X \subset \subset \Omega$. Then the comparison principle holds; that is,
if $u \in \USC(\overline{\Omega})$ and $v \in \LSC(\overline{\Omega})$ are $\Theta$-subharmonic and $\Theta$-superharmonic respectively in $\Omega$, then
\[
\mbox{$u \leq v$ on $\partial \Omega \ \ \Rightarrow \ \ u \leq v$ in $\Omega$.}
\]
\end{thm}

This comparison principle might be useful in order to obtain comparison principles fro certain PDEs; for example,  see Remark \ref{cut_example}. Its proof is based on two observations. First, under the assumption \eqref{cont_cut},  one can define for all $M > 0$ a uniformly continuous map $\Theta_M$ on $X \subset \subset \Omega$ that agrees with $\Theta$ for values in $[-M,M] \times \cS(N)$ in the codomain. Second, $\Theta$-subharmonic functions that are bounded in the sup-norm by $M$ are $\Theta_{M}$-subharmonic. One then concludes by a standard truncation argument and the comparison principle for continuous elliptic maps on $\Omega$, which are locally uniformly continuous.

Let
\[
\Theta_M(x) := \{(r,A) : (\psi_M(r), A) \in \Theta(x)\} \quad \forall x \in X, M > 0,
\]
where
\[
\psi_M(r) :=
\begin{cases}
r & \text{if $|r| \le M$}, \\
M & \text{if $|r| > M$}.
\end{cases}
\]
By the fact that $\psi_M$ is continuous and odd, it is straightforward to check that
\[
\widetilde{\Theta_M}(x) = \{(r,A) : (\psi_M(r), A) \in \widetilde \Theta(x)\} \quad \forall x \in X, M > 0.
\]

\begin{lem}\label{thetaMhaus} For all $M > 0$ large, $\Theta_M$ defined above is a proper and uniformly continuous elliptic map on each $X \subset \subset \Omega$.
\end{lem}

\begin{proof} First, we observe that $\Theta_M(x)$ is non-empty. Indeed, by Proposition \ref{prop:TSH_properties} (5), there exists $\tau>0$ such that $(-\tau, \tau I) \in \Theta(x)$ for all $x \in X$. Hence, for all $M \ge \tau$, $\psi_M(-\tau) = -\tau$, so $(-\tau, \tau I) \in \Theta_M(x)$ for all $x$. To prove that $\Theta_M(x) \neq \R \times \cS(N)$, one argues similarly via a couple $(r, A) \notin \Theta(x)$ for all $x$. Proper ellipticity easily follows from the monotonicity of $\psi_M$ and the degenerate ellipticity of $\Theta(x)$.
	
	Continuity of $\Theta$ can be obtained by the alternative characterization stated in Proposition \ref{UCHD}. Fix any $\eta > 0$, and from \eqref{cont_cut} let $\delta = \delta(R, \eta)$ be such that
	\[\text{$x,y \in \Omega$, $|x-y| < \delta$ $\Rightarrow$ $\Theta(x) \cap \left( [-R,R] \times \cS(N) \right) + (0, \eta I) \subset \Theta(y)$.}\]
	Let $(r, A) \in \Theta_M(x)$, so $(\psi_M(r), A) \in \Theta(x)$. Note that $|\psi_M(r)| \le M$, so $(\psi_M(r), A) \in  \Theta(x) \cap \left( [-R,R] \times \cS(N) \right) $. Hence, for $|x-y| < \delta$,
	\[
	(\psi_M(r), A) + (0, \eta I)\in \Theta(y),
	\]
	that in turn gives $(r, A + \eta I) \in \Theta_M(y)$. By ellipticity of $\Theta_M(y)$, one has $(r - \eta, A + \eta I) \in \Theta_M(y)$, that finally yields $\Theta_M(x) + (-\eta, \eta I) \subset \Theta_M(y)$ for all $|x-y| < \delta$.
\end{proof}

\begin{lem}\label{subh_preservation} Let $u \in \USC(\overline{X})$ be such that $|u| \le R$ for some $R > 0$. If $u$ is $\Theta$-subharmonic ($\widetilde{\Theta}$-subharmonic), then $u$ is $\Theta_{R}$-subharmonic ($\widetilde{\Theta_{R}}$-subharmonic).
\end{lem}

\begin{proof} Fix any $x_0 \in X$, and let $\varphi$ be $C^2$ near $x_0$, $\varphi(x_0) = u(x_0)$, and $u-\varphi$ have a local maximum at $x_0$. Since $|u(x_0)| \le R$, $\varphi(x_0) = \psi_R(\varphi(x_0))$, hence
	\[
	(\psi_R(\varphi(x_0)), D^2 \varphi(x_0)) = (\varphi(x_0), D^2 \varphi(x_0)) \in \Theta(x),
	\]
	that gives $J^{+}_{x_0} u \in \Theta_R(x)$. The proof for $\widetilde{\Theta_{R}}$-subharmonic functions is completely analogous. \end{proof}

We conclude with the proof of the comparison principle.

\begin{proof}[Proof of Theorem \ref{comparison_cut}] Arguing as in Lemma \ref{lem:AAT_reduction_BB}, with $X \subset \subset \Omega$ arbitrary, it is enough to consider a pair of functions $u,v  \in \USC(\overline{X})$ which are bounded from below (and above). Hence, we assume that for some $R > 0$, $|u|, |v| \le R$ on $\Omega$. By Lemma \ref{subh_preservation}, $u$ and $v$ are $\Theta_{R}$-subharmonic and $\widetilde{\Theta_{R}}$-subharmonic respectively. Since $\Theta_{R}$ is a uniformly continuous elliptic map on $X$ by Lemma \ref{thetaMhaus}, by the subharmonic addition theorem \ref{thm:AAT} $u - v \in \QSHD(X)$ for each $X \subset \subset \Omega$ and hence $u - v \in \QSHD(\Omega)$. The comparison principle for $u,v$ on $\Omega$ then follows from the zero maximum principle (Theorem \ref{thm:ZMP}).
	
\end{proof}

\section{Comparison principles for admissible solutions of proper elliptic PDEs}\label{sec:examples}

Armed with the potential theoretic comparison principle for continuous proper elliptic maps developed in previous sections, we derive comparison principles for some fully nonlinear second order PDEs. The equations we treat will have variable coefficients and will be gradient-free and {\em proper elliptic} (which, in general, may require the imposition of an {\em admissibility constraint} in order to ensure the needed monotonicity). The strategy we employ will be to determine structural conditions on the defining operator $F$ for the PDE which allow us to define a proper elliptic map $\Theta$ whose subharmonics/ superharmonics correspond to viscosity subsolutions/supersolutions of the PDE (with perhaps admissibility constraints on the upper and lower test jets used in the viscosity formulation). We call this the {\em correspondence principle} (see Theorem \ref{thm:SHCVS}). That being done, an additional condition will be placed on the operator $F$ in order to ensure that $\Theta$ is continuous. This additional structural condition involves some mild regularity and strict monotonicity assumptions on $F$ (see property \eqref{UCF}). Hence the comparison principle for the PDE follows directly from the comparison principle for continuous proper elliptic maps $\Theta$ given in Theorem \ref{thm:CP}.

While we have no complete recipe to associate a (continuous) proper elliptic map $\Theta$ to any given operator $F$, we are able to complete the program described above for a large class of equations that enjoy suitable monotonicity properties on proper elliptic {\it subsets} of $\R \times \cS(N)$. We call this the {\em constrained case} and this will be developed in Section \ref{proper_branches} below. Moreover, in the {\em unconstrained case}, when no admissibility constraint is needed, we will show that a natural choice of $\Theta$ can be made so that the correspondence principle holds without admissibility constraints on the upper and lower test jets in the (standard) viscosity formulation (see Remark \ref{rem:UC_Case}). Finally, we will present some comparison principles for two examples of fully nonlinear PDEs (one constrained and one unconstrained), to illustrate how our general theory applies in specific situations.

\subsection{Proper elliptic branches and admissible viscosity solutions of PDEs}\label{proper_branches}

We begin with the notion of proper ellipticity for a nonlinear equation
\begin{equation}\label{FNLE}
F(x, u(x), D^2u(x)) = 0, \quad x \in \Omega
\end{equation}
where $F: \Omega \times  \R \times \cS(N) \to \R$ is a continuous function satisfying
\begin{equation}\label{GammaNE}
\mbox{$\Gamma(x) := \{ (r,A) \in \R \times \cS(N): \ F(x,r,A) = 0 \} \neq \emptyset$ \ \ for each $x \in \Omega$.}
\end{equation}
We will call such an $F$ a {\em gradient-free operator}.

\begin{defn}\label{PE_branch} 
	Let $F$ be a gradient-free operator. The equation \eqref{FNLE} determined by $F$ is
	said to be  {\em proper elliptic} if there exists a proper elliptic map $\Theta: \Omega \to \wp(\R \times\cS(N))$ such that
	\begin{equation}\label{eq:branch1}
	\partial \Theta(x) \subset \Gamma(x), \quad x \in \Omega.
	\end{equation}
	In that case one calls the differential inclusion
	\begin{equation}\label{eq:branch2}
	J_xu = (u(x), D^2u(x)) \in \partial \Theta(x) , \quad x \in \Omega
	\end{equation}
	a {\em proper elliptic branch} of the equation \eqref{FNLE} {\em defined by $\Theta$}.
	\end{defn}

Notice that the definition depends only on the sets $\Gamma(x)$ and not on the particular form of the operator $F$, which was the insight of Krylov \cite{Kv95} for his general notion of ellipticity. Recall that $\Theta$ is a proper elliptic map if for each $x \in \Omega$
\begin{equation}\label{PEM1}
\Theta(x) \subsetneq \R \times \cS(N)\ \text{ is closed and non-empty}
\end{equation} 
and   
\begin{equation}\label{PEM2}
\Theta(x) + \cQ \subset \Theta(x) \ \ \text{where} \ \cQ = \cN \times \cP = \{(r,A) \in \R \times \cS(N): r \leq 0, A \geq 0\}.
\end{equation}
For gradient-free equations, ellipticity in the sense of Krylov requires only the weaker monotonicity assumption $\Theta(x) + (\{0\} \times \cP) \subset \Theta(x)$ for each $x \in \Omega$. Notice also that a given $F$ may admit many branches as $\Theta$ need not be unique. See section 2 of \cite{Kv95} for a discussion of this point.

Next, we turn to the question of structural conditions on $F$ for which proper elliptic branches can be defined. We will start by asking that at least some proper elliptic map $\Phi$ exists along which $F$ is proper elliptic ($\cQ$-monotone). One might think of the maximal such proper elliptic map $\Phi$. Subsequently, we will examine further conditions on the pair $(F, \Phi)$  for which there is a natural proper elliptic map $\Theta$ which determines a proper elliptic branch of the equation \eqref{FNLE}.  

\begin{defn}\label{defn:PEO} Let $F$ be a gradient-free operator. We say that $F$ is {\em proper elliptic} if there exists a proper elliptic map $\Phi: \Omega \to \wp(\R \times\cS(N))$ such that 
		\begin{equation}\label{PEO}
	F(x,r + s, A + P) \geq F(x,r,A) \ \ \forall x \in \Omega, (r,A) \in \Phi(x), (s,P) \in \cQ.
	\end{equation}
	In this case $(F,\Phi)$ will be called a \em{proper elliptic pair}.
	\end{defn}

Given a proper elliptic pair $(F, \Phi)$, in general, $\Phi$ will not satisfy the {\em branch condition} $\partial \Phi(x) \subset \Gamma(x)$. We will examine one general situation in which a suitable subset $\Theta(x)$ of $\Phi(x)$ for each $x \in \Omega$ does indeed determine a proper elliptic branch of \eqref{FNLE}.   

Before stating a general result, a simple example is instructive. Consider the following Monge-Amp\`ere equation 
\begin{equation}\label{MAE}
 -u(x) \det(D^2(x)) = f(x), \ \ x \in \Omega,
\end{equation}
where $f$ is continuous and nonnegative. The operator $F(x,r,A) = -r \det(A) - f(x)$ is clearly $\cQ$ monotone on all of $\cQ = \cN \times \cP$; that is, for each $x \in \Omega, (r,A) \in \cQ$ and $(s,P) \in \cQ$ one has
$$
F(x,r + s, A + P) = -(r + s) \det(A + P) - f(x) \geq -r \det(A) - f(x) = F(x,r,A).
$$
Hence for the proper elliptic map defined by $\Phi(x) = \cQ$ for each $x \in \Omega$ one has that $(F, \Phi)$ a proper elliptic pair. In addition, it is clear that this constant map $\Phi$ is the maximal map for which $F$ restricted to $\Phi$ is $\cQ$-monotone. Now, for each $x \in \Omega$ one has
$$ 
\Gamma(x) = \{ (r,A) \in \R \times \cS(N): -r \det(A) = f(x) \}
$$
while 
$$
\partial \Phi(x) = \{ (r,A) \in \cN \times \cP: r = 0 \ \ \text{or} \ A \in \partial \cP \}
$$
and hence $F(x,r,A) = f(x)$ for each $(r,A) \in \partial \Phi(x)$ and the branch condition $\partial \Phi(x) \subset \Gamma(x)$ holds only at points where $f(x) = 0$. This suggests reducing $\Phi$ to
$$
\Theta(x):= \{(r,A) \in \cQ: \ -r \det(A) - f(x) \geq 0 \},
$$
where one easily checks that $(F, \Theta)$ is a proper elliptic pair and that $\partial \Theta(x) \subset \Gamma(x)$ so that $\Theta$ defines an elliptic branch of \eqref{MAE}.

We now give the general statement suggested by this example, where we recall that $\Gamma(x) = \{(r,A) \in \R \times \cS(N): F(x,r,A) = 0\}$.

\begin{thm}[Proper elliptic branches]\label{pick_branch} Let $(F, \Phi)$ be a proper elliptic pair; that is, the gradient-free operator is $\cQ$-monotone when restricted to the proper elliptic map $\Phi$ in the sense \eqref{PEO}. Assume that the following two conditions hold: 
\begin{equation}\label{PB1}
	\Phi(x) \cap \Gamma(x) \neq \emptyset \ \ \text{for each} \ x \in \Omega;
	\end{equation}
	\begin{equation}\label{PB2}
	\partial \Phi(x) \subset \{ (r,A) \in \R \times \cS(N): F(x,r,A) \leq 0\}.
	\end{equation}
	 Then, the map $\Theta: \Omega \to \wp(\R \times \cS(N))$ defined by
	\begin{equation}\label{def_branch1}
	\Theta(x) := \{ (r,A) \in \Phi(x):  \ F(x,r,A) \geq 0 \}
	\end{equation}
	is a proper elliptic map and $\Theta$ defines a proper elliptic branch of the PDE \eqref{FNLE} determined by $F$; that is,
\begin{equation}\label{branch_condition}
	\partial \Theta(x) \subset \Gamma(x)
	\end{equation}
	\end{thm}

\begin{proof} This is a generalization of \cite[Proposition 5.1]{CP17}. For each $x \in \Omega$, $\Theta(x) \neq \emptyset$ by the first condition \eqref{PB1} and is not all of $\R \times \cS(N)$ since $\Phi(x)$ is a proper subset (by Definition \ref{proper_ell_map}). Moreover $\Theta(x)$ is closed since $\Phi(x)$ is closed and $F$ is continuous, where it would suffice to have $F(x, \cdot, \cdot)$ upper semicontinuous for each $x$ fixed. Hence $\Theta$ satisfies the property \eqref{PE1} of a proper elliptic map. For the $\cQ$-monotonicity condition \eqref{PE2}, notice that for each $x \in \Omega$ and for each $(r,A) \in \Theta(x) \subset \Phi(x)$ one has
$$
\mbox{$(r,A) + (s,P) \in \Phi(x)$, \quad for each $s \leq 0$ and $P \geq 0$}
$$
by the $\cQ$-monotonicity property of $\Phi$. Using the $\cQ$-monotonicity of $F$ restricted to $\Phi$ \eqref{PEO} and the definition \eqref{def_branch1} of $\Theta$ one has
$$
F(x, r + t, A + P) \geq F(x,r,A) \geq 0,
$$
and hence $\Theta$ is a proper elliptic map.

It remains only to check that $\Theta$ defines a branch; that is, that \eqref{branch_condition} holds. One easily checks that $\partial \Theta(x)$ is the union of two sets
\begin{equation}\label{EBPhi2}
  \partial \Phi(x) \cap \{ (r,A) \in \R \times \cS(N): F(x,r, A) \geq 0 \}
\end{equation}
and
\begin{equation}\label{EBPhi3}
\Phi(x) \cap \{ (r,A) \in \R \times \cS(N): F(x,A) = 0 \},
\end{equation}
which yields \eqref{branch_condition} if the branch condition \eqref{PB2} holds since $\Phi(x)$ is closed by definition.
\end{proof}

In addition to the equation \eqref{MAE} mentioned above, examples of equations for which this proposition applies include elliptic equations with $F = F(x,A)$ independent of $r$ as treated in \cite{CP17}. Additional examples, where one also has the comparison principle, will be given in the next subsection.

\smallskip

Now that a proper elliptic branch of the PDE (constrained by $\Phi$) is defined by the map $\Theta$ in \eqref{def_branch1}, we turn to the definition of $\Phi$-admissible viscosity sub/supersolutions to the equation \eqref{FNLE}, with the idea of establishing the equivalence between such admissible sub/supersolutions and $\Theta$-sub/superharmonics (Definition \ref{defn:TSH}) for the map $\Theta$. We will again make use of the {\em upper and lower test jets}, which we recall are defined for each fixed $x_0 \in \Omega$ by
\begin{equation}\label{UTJ}
J^{+}_{x_0} u := \{ (\varphi(x_0),D^2 \varphi(x_0)):  \varphi \ \text{is} \ C^2 \ \text{near} \ x_0, \  u \leq \varphi \ \text{near $x_0$ with equality in} \ x_0 \}
\end{equation}
and
\begin{equation}\label{LTJ}
J^{-}_{x_0} u := \{ (\varphi(x_0),D^2 \varphi(x_0)):  \varphi \ \text{is} \ C^2 \ \text{near} \ x_0, \  u \geq \varphi \ \text{near $x_0$ with equality in} \ x_0 \}.
\end{equation}

\begin{defn}\label{Vs_def} Let $F: \Omega \times \R \times \cS(N) \to \R$ be continuous and $\Phi: \Omega \to \wp(\R \times \cS(N))$ a proper elliptic map. 
\begin{itemize}
\item[(a)] One says that $u \in \USC(\Omega) $ is a {\em $\Phi$-admissible viscosity subsolution of \eqref{FNLE} in $\Omega$} if for every $x_0 \in \Omega$  one has
\begin{equation}\label{Vss}
(r,A) \in J^{+}_{x_0} u \ \ \Rightarrow \ \  F(x_0,r,A) \geq 0 \ \ \text{and} \ \ (r,A) \in \Phi(x_0)
\end{equation}
\item[(b)] One says that $u \in \LSC(\Omega) $ is a {\em $\Phi$-admissible viscosity supersolution of \eqref{FNLE} in $\Omega$} if for every $x_0 \in \Omega$ one has
\begin{equation}\label{VSs}
(r,A) \in J^{-}_{x_0} u \ \ \Rightarrow \ \  F(x_0,r,A) \leq 0 \ \ \text{or} \ \ (r,A) \not\in [\Phi(x_0)]^{\circ}
\end{equation}
\end{itemize}
One says that $u \in C(\Omega)$ a {\em $\Phi$-admissible viscosity solution of \eqref{FNLE} in $\Omega$} if both conditions (a) and (b) hold.
\end{defn}

Notice that for $\Theta(x):= \{(r,A) \in \Phi(x): \ F(x,r,A) \geq 0 \}$, the $\Phi$-admissible subsolution condition \eqref{Vss} is equivalent to
\begin{equation}\label{Tsx}
	J^{+}_{x_0} u \subset \Theta(x_0),
\end{equation}
which defines $u \in \USC(\Omega)$ being $\Theta$-subharmonic in $x_0$ (see Definition \ref{Vs_def}). On the other hand, recall that $u \in \LSC(\Omega)$ is $\Theta$-superharmonic in $x_0$ if 
\begin{equation}\label{TSx}
J^{-}_{x_0} u \subset \left(  [\Theta(x_0)]^{\circ} \right)^c.
\end{equation}
Under an additional hypothesis of {\em non degeneracy}, the condition \eqref{TSx} is equivalent to the condition \eqref{VSs}.

\begin{thm}[Correspondence principle]\label{thm:SHCVS} Let $(F, \Phi)$ be a proper elliptic pair and let $\Theta$ be the corresponding proper elliptic map defined by \eqref{def_branch1}; that is,
$$
	\Theta(x) := \{ (r,A) \in \Phi(x): \ \ F(x,r,A) \geq 0 \}.
$$
 Then the following equivalences hold.
	\begin{itemize}
		\item[(a)] A function $u \in \USC(\Omega)$ is a $\Phi$-admissible viscosity subsolution of \eqref{FNLE} in $\Omega$ if and only if $u \in \TSH(\Omega)$ ($u$ is $\Theta$-subharmonic in $\Omega$).
		\item[(b)] A function $u \in \LSC(\Omega)$ is a $\Phi$-admissible viscosity supersolution of \eqref{FNLE} in $\Omega$ if and only if $-u \in \TSHD(\Omega)$ ($u$ is $\Theta$-superharmonic in $\Omega$) 
		\underline{provided} that the following non-degeneracy condition is satisfied:
		\begin{equation}\label{NDC}
		\mbox{ $ F(x,r,A) > 0$ for each $x \in \Omega$ and each $(r,A) \in [\Theta(x)]^{\circ}$.}
		\end{equation}
	\end{itemize}
\end{thm}

\begin{proof} The equivalence of part (a) has been noted above. For the equivalence of part (b), notice that the $\Phi$-admissible supersolution condition \eqref{VSs} is clearly equivalent to
\begin{equation}\label{VS1}
J^{-}_{x_0} u \subset \left([\Phi(x_0)]^{\circ}\right)^c \cup \{ (r,A) \in \R \times \cS(N): \ F(x_0,r,A) \leq 0 \}
\end{equation}
and by comparing \eqref{VS1} with \eqref{TSx}, it suffices to show that for each fixed $x \in \Omega$ one has
\begin{equation}\label{VS2}
\left([\Theta(x)]^{\circ} \right)^c = \left([\Phi(x)]^{\circ}\right)^c \cup \{ (r,A) \in \R \times \cS(N): \ F(x,r,A) \leq 0 \}.
\end{equation}
Making use of the duality $\wt{\Theta}(x) := - \left([\Theta(x)]^{\circ} \right)^c$, by negating the elements in \eqref{VS2}, it suffices to show that 
\begin{equation}\label{VS3}
\wt{\Theta}(x) =  \Psi(x) \ \ \text{where} \ \ \Psi(x):=\wt{\Phi}(x) \cup \{ (r,A) \in \R \times \cS(N): \ F(x,-r,-A) \leq 0 \} .
\end{equation}
We calculate directly the dual $\wt{\Theta}(x)$ using the property $[\widetilde{\Theta}(x)]^{\circ} = - [ \Theta(x) ]^{c}$ which follows from the reflexivity of $\Theta(x)$. By definition, we have
\begin{equation*}\label{calc_dual1}
	\Theta(x) = \Phi(x) \cap \{ (r,A) \in \R \times \cS(N): \ F(x,r,A) \geq 0 \}, \ \ x \in \Omega,
	\end{equation*}
and hence
\begin{equation*}\label{calc_dual2}
[ \Theta(x) ]^{c}= [ \Phi(x) ]^{c} \cup \{ (r,A) \in \R \times \cS(N): \ F(x,r,A) < 0 \}, \ \ x \in \Omega.
\end{equation*}
Hence
\begin{equation*}\label{calc_dua13}
-[ \Theta(x) ]^{c}= [\wt{\Phi}(x)]^{\circ}  \cup \{ (r,A) \in \R \times \cS(N): \ F(x,-r,-A) < 0 \}, \ \ x \in \Omega,
\end{equation*}
which yields
\begin{equation*}\label{VS4}
[ \wt{\Theta}(x) ]^{\circ}= [\wt{\Phi}(x)]^{\circ}  \cup \{ (r,A) \in \R \times \cS(N): \ F(x,-r,-A) < 0 \}, \ \ x \in \Omega
\end{equation*}
Now we take the closure in $\R \times \cS$ using the property \eqref{closure_interior} for the proper elliptic maps $\wt{\Theta}$ and $\wt{\Phi}$ to conclude
\begin{equation}\label{VS4}
 \wt{\Theta}(x) = \wt{\Phi}(x)  \cup \overline{ \{ (r,A) \in \R \times \cS(N): \  F(x,-r, -A) < 0 \}}, \ \ x \in \Omega
\end{equation}
We just need to check now that $\Psi(x)$ as defined in \eqref{VS3} equals $\Theta(x)$ as calculated in \eqref{VS4}. By the continuity of $F$ one clearly has $\wt{\Theta}(x) \subset \Psi(x)$. For the reverse inclusion it suffices to show that
\begin{equation}\label{VS5}
(r,A) \in \R \times \cS(N) \ \text{with} \ F(x,-r,-A) \leq 0 \ \ \Rightarrow \ \ (r,A) \in \wt{\Theta}(x).
\end{equation}
By the non degeneracy condition \eqref{NDC}, $(-r,-A) \notin [\Theta(x)]^{\circ}$ which by duality means $(r,A) \in - \left( [\Theta(x)]^{\circ} \right)^c = \wt{\Theta}(x)$, as needed.
\end{proof}

A few observations about this correspondence principle are in order.

\begin{rem}\label{rem:CEP}
	For $F, \Phi$ and $\Theta$ as in Theorem \ref{thm:SHCVS}, if one also assumes the branch condition \eqref{PB2}, then $\Theta$ defines a proper elliptic branch of the PDE \eqref{FNLE} by Theorem \ref{pick_branch}. Moreover, by adding in the non degeneracy condition \eqref{NDC} (which is not required for the Correspondence Principle of \ref{thm:SHCVS}) it follows that 
	\begin{equation}\label{CP1} 
	\partial \Theta(x) = \{ (r,A) \in \Theta(x): \ F(x,r,A) = 0 \} = \Theta(x) \cap \Gamma(x), \ \ \forall \, x \in \Omega
	\end{equation}
	and
		\begin{equation}\label{CP2}
		c_0: = \inf_{(r,A) \in \Theta(x)} F(x, r, A) \ \text{is finite} \ (c_0 = 0).
		\end{equation}
	Hence, borrowing the terminology of \cite{CHLP19} in the constant coefficient setting, we can say that  $(F, \Theta)$ is a {\em (constrained case) compatible proper elliptic pair} for which the correspondence principle holds. 
	\end{rem}

Next, we briefly discuss the ``standard'' case in which there is no a priori need to impose admissibility constraints. In this case we will derive a correspondence principle between standard viscosity subsolutions (supersolutions) of proper elliptic operators and $\Theta$-subharmonic (superharmonic) functions under mild non-degeneracy conditions (see \eqref{Theta_UC} and \eqref{F_UC}). 

\begin{rem}[The unconstrained case]\label{rem:UC_Case} If $F$ is a gradient-free operator ($F$ continuous with $\Gamma(x) \neq \emptyset$) which is proper elliptic on all of $\R \times \cS(N)$; that is,
	$$
	F(x, r+s,A+P) \geq F(x,r,A) \ \ \forall \, x \in \Omega, (r,A) \in \R \times \cS(N), (s,P) \in \cQ,
	$$
	then there is no need to constrain $F$ to some proper subset of the gradient-free jet space $\cJ:= \R \times \cS(N)$ in order for $F$ to be $\cQ$-monotone. In this case, by letting $\Phi = \cJ$ in Definition \ref{Vs_def} one recovers usual notion of viscosity subsolutions and supersolutions since the condition $(r,A) \in \cJ$ in \eqref{Vss} holds trivially and the possibility $(r,A) \notin [\cJ]^{\circ} = \cJ$ cannot occur. We will say that $(F, \cJ)$ is an {\em (unconstrained case) compatible proper elliptic pair}.
	
	Now, if one defines the map $\Theta: \Omega \to \wp(\R \times \cS(N))$ as before with $\Phi \equiv \cJ$; that is, 
	\begin{equation}\label{Theta_UC}
	\Theta(x) := \{ (r,A) \in \cJ: \ F(x,r,A) \geq 0\},
\end{equation}
	then $\Theta(x)$ will be closed by the continuity of $F$ and non empty as $\Gamma(x) \neq \emptyset$. In addition, $\Theta(x)$ will be a proper subset of $\cJ$ if 
	\begin{equation}\label{F_UC}
\mbox{ the fiber $\{ (r,A) \in \cJ: F(x, r, A) < 0 \}$ is not empty for each $x \in \Omega$}.
\end{equation} 
Hence $\Theta$ will be a proper elliptic map since for each $x \in \Omega$, the $\cQ$-monotonicity of $F$ on all of $\cJ$ yields the $\cQ$-monotonicity of $\Theta(x)$. Finally, in this case, the non-degeneracy condition \eqref{NDC} becomes
	\begin{equation}\label{Theta_ND}
\partial \Theta(x) = \{ (r,A) \in \Theta(x): \ F(x,r,A) = 0\} = \Gamma(x)
\end{equation}
and so $(F, \Theta)$ is a compatible pair in the sense \eqref{CP1}-\eqref{CP2}. Hence
one has correspondence principle between $\Theta$-superharmonic functions and standard (unconstrained) viscosity supersolutions of the PDE \eqref{FNLE}.
\end{rem}

\subsection{Comparison principles for PDEs from potential theoretic comparison }

In the previous subsection, we have discussed fiberwise properties (i.e., for $x \in \Omega$ fixed) of the operator $F(x,\cdot,\cdot)$ that ensure that the map $\Theta$ (defined by \eqref{def_branch1} in the constrained case and by \eqref{Theta_UC} in the unconstrained case) is: 1) proper elliptic, 2) defines a proper elliptic branch of the PDE \eqref{FNLE} and 3) satisfies the correspondence principle (of Theorem \eqref{thm:SHCVS} in the constrained case and of Remark \ref{rem:UC_Case} in the unconstrained case). 
We now discuss structural conditions on $F$ as $x$ varies which will ensure that the associated proper elliptic map $\Theta$ is continuous. Combining this continuity with the correspondence principle of Theorem \ref{thm:SHCVS} will then yield the validity of the comparison principle for ($\Phi$-admissible) viscosity solutions of the PDE \eqref{FNLE}, by applying the potential theoretic version of comparison (Theorem \ref{thm:CP}) for continuous proper elliptic maps. 
\begin{thm}[Continuity of proper elliptic maps]\label{UCbranch} Let $F \in C(\Omega \times \R \times \cS(N))$ be a gradient-free operator and $\Theta: \Omega \to \wp(\R \times \cS(N))$ be a proper elliptic map of the form
	\begin{equation}\label{Theta_def}
	\Theta(x) := \{ (r,A) \in \Phi(x): \ F(x,r,A) \geq 0 \}
	\end{equation}
	where either
	\begin{equation}\label{CPhi}
	\mbox{$\Phi$ is a continuous proper elliptic map on $\Omega$ \ \ (constrained case)}
	\end{equation}
	or
	\begin{equation}\label{UCPhi}
	\mbox{$\Phi(x) = \cJ = \R \times \cS(N)$ for each $x \in \Omega$ \ \ (unconstrained case).}
	\end{equation}
Assume that the pair $(F, \Phi)$ satisfies the following regularity condition: for every $\Omega' \subset \subset \Omega$ and for every $\eta > 0$, there exists $\delta= \delta(\eta, \Omega') > 0$ such that 
	\begin{equation}\label{UCF}
	F(y, r - \eta, A + \eta I) \geq F(x, r, A) \quad \forall (r, A) \in \Phi(x), \forall x,y \in \Omega' \ {\rm with \ } |x - y| < \delta.
	\end{equation}
	Then, the proper elliptic map $\Theta$ is continuous.
\end{thm}

\begin{proof} 
	We will show that $\Theta$ is locally uniformly continuous. Since $\Theta$ is assumed to be proper elliptic, by Proposition \ref{UCHD}, it suffices to show that
	for every fixed $\Omega' \subset \subset \Omega$ and for every fixed $\eta > 0$, there exists $\delta_{\Theta}= \delta_{\Theta}(\eta, \Omega') > 0$ such that for each $x,y \in \Omega'$
	\begin{equation}\label{LUC_Theta}
	\mbox{$|x -y| < \delta_{\Theta} \ \Rightarrow \ \Theta(x) + (- \eta, \eta I) \subset \Theta(y)$.}
	\end{equation}
	In the constrained case \eqref{CPhi}, we have the validity of \eqref{LUC_Theta} with $\Phi$ in place of $\Theta$ for some $\delta_{\Phi} = \delta_{\Phi}(\eta, \Omega)$. It suffices to choose $\delta_{\Theta} = \min \{ \delta_{\Phi}, \delta \}$. Indeed, for each pair $x,y \in \Omega'$ with $|x-y| < \delta_{\Theta}$, pick an arbitrary $(r,A) \in \Theta(x)$ so that $(r, A) \in \Phi(x)$ and $F(x,r,A) \geq 0$, which by the continuity of $\Phi$ and the regularity property \eqref{UCF} yields
	\begin{equation}\label{CT1}
	(r - \eta, A + \eta I) \in \Phi(y) \ \ \text{and} \ \ F(y, r - \eta, A + \eta I) \geq F(x, r, A)  \geq 0,
	\end{equation}
	which yields the inclusion in \eqref{LUC_Theta}. 
	
	In the unconstrained case \eqref{UCPhi}, the constant map $\Phi \equiv \cJ$ is trivially continuous (\eqref{LUC_Theta} for $\Phi$ holds for every $\delta_{\Phi} > 0$ and hence it suffices to choose $\delta_{\Theta} = \delta$ and use the regularity condition \eqref{UCF}.
\end{proof}	

Before moving on to comparison principles, a few remarks are in order. 

\begin{rem}\label{rem:cont0} In Theorem \ref{UCbranch}, the structural condition \eqref{UCF} on $F$ is merely sufficient to ensure that a proper elliptic map $\Theta$ given by \eqref{LUC_Theta} is continuous. The (locally uniform) continuity of $\Theta$ is equivalent to the statement that: for every $\Omega' \subset \subset \Omega$ and for every $\eta > 0$, there exists $\delta= \delta(\eta, \Omega') > 0$ such that $\forall x,y \in \Omega' \ {\rm with \ } |x - y| < \delta$ one has
	\begin{equation}\label{UCF_defn}
	F(x,r,A) \geq 0  \ \ \text{and} \ \ (r, A) \in \Phi(x) \ \ \Rightarrow \ \ F(y, r - \eta, A + \eta I) \geq 0.
	\end{equation}
This condition is weaker, in general, than the structural condition \eqref{UCF} and hence useful to keep in mid for specific applications (see, for example, the proof of Theorem \ref{thm:CSLPE}). On the other hand, the structural condition \eqref{UCF} can be more easily compared to other structural conditions on $F$ present in the literature.	
	\end{rem}

\begin{rem}\label{rem:cont1}
	In Theorem \ref{UCbranch} we have \underline{assumed} that $\Theta$ defined by \eqref{Theta_def} is a proper elliptic map. By what we have done previously, we have sufficient conditions which guarantee that $\Theta$ is indeed proper elliptic. In particular, it is enough to assume that
	\begin{equation}\label{conti1}
	F(x,r+s,A + P) \geq F(x,r,A) \quad \forall x \in \Omega, (r, A) \in \Phi(x), (s,P) \in \cQ = \cN \times \cP 
	\end{equation}
	in order to ensure the needed $\cQ$-monotonicity of each $\Theta(x)$. Each $\Theta(x)$ will be closed (by the continuity of $F$) and non-empty provided
	\begin{equation}\label{conti2}
		\Gamma(x) \cap \Phi(x) \neq \emptyset \ \ \text{for each} \ x \in \Omega.
	\end{equation}
	The remaining condition $\Theta(x) \subsetneq \R \times \cS(N)$ is always satisfied in the constrained case $(\Theta(x) \subset \Phi(x) \subsetneq \R \times \cS(N)$). In the unconstrained case, one need only assume
		\begin{equation}\label{conti3}
	\{ (r,A) \in \cJ: \ F(x,r,A) < 0 \} \neq \emptyset \ \  \text{for each} \ x \in \Omega.
	\end{equation}
		\end{rem}

\begin{rem}\label{rem:conti2}
	The importance of having $\Theta$ be proper elliptic on $\Omega$ is twofold. On the one hand, we can exploit the formulation \eqref{uusc2} for proper elliptic maps, which makes the regularity condition \eqref{UCF} a natural one. On the other hand, proper ellipticity on all of $\Omega$ rules out the possibility that $\Theta (x) = \R \times \cS(N)$ on some proper subset $\Omega'$ of $\Omega$ but with $\Theta$  proper elliptic on $\Omega\setminus \Omega'$. In such a case, by picking any $x \in \Omega'$ and $y \in \Omega\setminus \Omega'$, since $\Theta(y) \subsetneq \R \times \cS(N)$, $d_{\mathcal{H}}(\R \times \cS(N), \Theta(y)) = +\infty$, as noted in \eqref{HDall}. This holds for pairs $x,y$ which are arbitrarily close.
\end{rem}

\begin{rem}\label{rem:conti3}
As a final comment, we note that when the pair $(F, \Phi)$ is proper elliptic, the condition \eqref{UCF} is restrictive only for $\eta > 0$ small. Indeed, if for some $\eta^* > 0$, and for each $\eta \in (0, \eta^*]$, there exists $\delta = \delta(\eta, \Omega) > 0$ such that \eqref{UCF} holds, the proper ellipticity  \eqref{conti1} implies that \eqref{UCF} continues to hold for each $\eta > \eta^*$ by taking $\delta(\eta, \Omega') = \delta(\eta^*, \Omega')$. 

\end{rem}
	
We conclude this subsection by stating a comparison principle for viscosity solutions of PDE \eqref{FNLE} in both constrained and unconstrained cases. 

\begin{thm}[Comparison principle: PDE version]\label{thm:CPpde} Given $F \in C(\Omega \times \R \times \cS(N), \R)$ and $\Phi: \Omega \to \wp(\R \times\cS(N))$ where either 
	\begin{equation}\label{TCP1}
\mbox{$\Phi$ is a continuous proper elliptic map on $\Omega$ \ \ (constrained case)}
\end{equation}
or
\begin{equation}\label{TCP2}
\mbox{$\Phi(x) = \cJ = \R \times \cS(N)$ for each $x \in \Omega$ \ \ (unconstrained case)}
\end{equation}
and let $\Theta: \Omega \to \wp(\R \times \cS(N))$ be defined by
\begin{equation}\label{TCP3}
\Theta(x) := \{ (r,A) \in \Phi(x): \ F(x,r,A) \geq 0 \}.
\end{equation}
Assume that $F$ restricted to $\Phi$ is proper elliptic; that is,
\begin{equation}\label{TCP4}
F(x,r + s, A + P) \geq F(x,r,A) \ \ \forall x \in \Omega, (r,A) \in \Phi(x), (s,P) \in \cQ
\end{equation}
and that the pair $(F, \Phi)$ satisfies the regularity property \eqref{UCF}; that is,
for every $\Omega' \subset \subset \Omega$ and for every $\eta > 0$, there exists $\delta= \delta(\eta, \Omega') > 0$ such that
\begin{equation}\label{TCP5}
F(y, r - \eta, A + \eta I) \geq F(x, r, A) \quad \forall (r, A) \in \Phi(x), \forall x,y \in \Omega' \ {\rm with \ } |x - y| < \delta.
\end{equation}
Assume the non-empty condition \eqref{PB1}; that is,
\begin{equation}\label{TCP6}
\Gamma(x) \cap \Phi(x) \neq \emptyset \ \ \text{for each} \ x \in \Omega
\end{equation}
and the non-degeneracy condition \eqref{NDC}; that is, 
	\begin{equation}\label{TCP7}
\mbox{ $ F(x,r,A) > 0$ for each $x \in \Omega$ and each $(r,A) \in [\Theta(x)]^{\circ}$.}
\end{equation}
Moreover, in the constrained case \eqref{TCP1} assume the branch condition \eqref{PB2}; that is,
	\begin{equation}\label{TCP8}
\partial \Phi(x) \subset \{ (r,A) \in \R \times \cS(N): F(x,r,A) \leq 0\}.
\end{equation}
and in the unconstrained case \eqref{TCP2} assume the properness condition \eqref{F_UC}; that is,
	\begin{equation}\label{TCP9}
\{ (r,A) \in \cJ: F(x, r, A) < 0 \} \neq \emptyset \ \ \text{for every} \ x \in \Omega.
\end{equation} 
Then, the map $\Theta$ is a continuous proper elliptic map and defines a proper elliptic branch of the PDE \eqref{FNLE};  that is,
\begin{equation}\label{TCP10}
F(x, u(x), D^2u(x)) = 0, \quad x \in \Omega
\end{equation}
and for every bounded domain $\Omega$ the comparison principle for the PDE \eqref{TCP10} holds; that is, 
\begin{equation}\label{TCP11}
\mbox{$u \leq v$ on $\partial \Omega \ \ \Rightarrow \ \ u \leq v$ in $\Omega$.}
\end{equation}
if $u$ is a $\Phi$-admissible viscosity subsolution of \eqref{TCP10} in $\Omega$ and  $u$ is a $\Phi$-admissible viscosity supersolution of \eqref{TCP10} in $\Omega$.
\end{thm}

\begin{proof} In the constrained case \eqref{TCP1}, the map $\Theta$ defined by \eqref{TCP3} is a proper elliptic map and defines a proper elliptic branch of \eqref{TCP10} by applying Theorem \ref{pick_branch}, where one uses \eqref{TCP4}, \eqref{TCP6} and \eqref{TCP8}. The non-degeneracy condition \eqref{TCP7} then yields the correspondence principle of Theorem \ref{thm:SHCVS}. Hence the comparison principle \eqref{TCP11} follows from the potential theoretic version of comparison (Theorem \ref{thm:CP}). 
	
	In the unconstrained case \eqref{TCP2}, the map $\Theta$ defined by \eqref{TCP3} is proper elliptic, as discussed in Remark \ref{rem:UC_Case}, where one uses \eqref{TCP4}, \eqref{TCP6} with $\Phi(x) = \cJ$ for each $x \in \Omega$ and the properness condition \eqref{TCP9}. Using the non-degeneracy condition \eqref{TCP7}, which in this case means \eqref{Theta_ND}, one has the correspondence principle between $\Theta$ subharmonics/superharmonics and standard viscosity subsolutions/supersolutions of the PDE \eqref{TCP10} (as noted in Remark \ref{rem:UC_Case}). Hence, again, comparison for the PDE \eqref{TCP10} reduces to the validity of Theorem \ref{thm:CP}.
\end{proof}

\subsection{Comparison in the constrained case.}

We now focus our attention on specific examples. We consider first the
validity of the comparison principle for an interesting prototype equation that is defined by an operator that is proper elliptic only when constrained to certain proper subsets of $\R \times
\cS(N)$. We will consider the equation 
\begin{equation}\label{cy}
[-u(x)]^{N+2} \det D^2 u(x) = h(x), \qquad x \in \Omega, \\
\end{equation}
where
\begin{equation}\label{Lphiass}
\mbox{$h \in C(\Omega)$ and $h \ge 0$ on $\Omega$.}
\end{equation}
When $h$ is positive constant, this kind of Monge-Amp\`ere equation is important in the question of the completeness of the affine metric of hyperbolic affine spheres as treated by Cheng and Yau in \cite{CY86}. In particular, for any negative constant $L$, a necessary and sufficient condition for the graph of $v$ to
be a hyperbolic affine sphere with {\em affine mean curvature} $L$ and center at the origin is that its Legendre transform $u=v^*$ satisfies (see section 5 of \cite{CY86}):
$$
\det D^2 u(x) = \left[ \frac{L}{u(x)} \right]^{N+2}, \qquad x \in \Omega, 
$$
which is equivalent to \eqref{cy} for $u > 0$ with $L = -h <0$. Here, we consider the case of $h$ being a function of
the $x$ variable, possibly vanishing on $\Omega$. Clearly,
\[
F(x, r, A) := (-r)^{N+2} \det A - h(x)
\]
fails to satisfy proper ellipticity conditions on the whole $\R \times
\cS(N)$. Still, $F$ restricted to $\cQ = \cN \times \cP$ satisfies
\eqref{TCP4}. We can prove the following comparison result.

\begin{thm}\label{thm:HASE} Suppose that $h$ satisfies \eqref{Lphiass}. Then, the map
	$\Theta: \Omega \to \wp(\R \times \cS(N))$ defined by
	\[
	\Theta(x) := \{ (r,A) \in \cQ :  \  (-r)^{N+2} \det A \geq h(x) \}
	\]
	is a continuous proper elliptic map and defines a proper elliptic
	branch of \eqref{cy}. Moreover, for any $\cQ$-admissible viscosity
	subsolution $u$ and any $\cQ$-admissible viscosity supersolution $v$ of
	\eqref{cy} (in the sense of Definition \ref{Vs_def}),
	\[
	\mbox{$u \leq v$ on $\partial \Omega \ \ \Rightarrow \ \ u \leq v$ in
		$\Omega$.}
	\]
	
\end{thm}

\begin{proof} To show that the comparison principle holds, it is
	sufficient to check all the assumptions of Theorem \ref{thm:CPpde} in
	the constrained case. First, $\cQ$ is $x$-independent, so it is clearly
	Hausdorff continuous. The monotonicity condition \eqref{TCP4} is easily verified:
	$F(x,r,A)$ is decreasing in $r$ and increasing in $A$ if $r \le 0$ and $ A \ge 0$; that is, if $(r,A) \in \cQ$. For all $x \in \Omega$, $(r,A) =
	\big(-(h(x))^{\frac1{N+2}}, I\big) \in \cQ$ satisfies $F(x,r,A) = 0$,
	hence the non-empty condition \eqref{TCP6} holds. Moreover, since $\partial \cQ = \left( \{0\}\times
	\cP \right) \cup \left( \cN\times\partial \cP \right)$, $F(x, \cdot, \cdot) \le 0$ on $\partial
	\cQ$ for each $x$, we have the branch condition \eqref{TCP8}. To check the regularity condition \eqref{TCP5}, note
	that for $\eta > 0$,
	\begin{multline*}
	F(y, r - \eta, A + \eta I) = (-r + \eta)^{N+2} \det (A + \eta I) - h(y)
	\ge \\
	(-r + \eta)^{N+2} \det A + \eta^N (-r + \eta)^{N+2} - h(y) \ge F(x, r,
	A) + \eta^{2N+2} - h(y) + h(x),
	\end{multline*}
	for all $(r,A) \in \cQ$ and $y \in \Omega$. For any fixed $\Omega'
	\subset \subset \Omega$, by the uniform continuity of $h$ on
	$\Omega'$, it is sufficient to pick $|x-y| < \delta$ and
	$\delta = \delta(\eta, \Omega')>0$ small enough to have $\eta^{2N+2} -
	h(y) + h(x) \ge 0$. Similarly, the non-degeneracy condition \eqref{TCP7}
	is satisfied. Indeed, for any $(r,A) \in [\Theta(x)]^{\circ}$, $(r+
	\eta,A -\eta I ) \in \Theta(x)$ for small $\eta > 0$. As before,
	\[
	F(x, r, A) = F(x, r + \eta - \eta, A - \eta I + \eta I) \ge F(x, r +
	\eta, A - \eta I) + \eta^{2N+2} > 0,
	\]
	The comparison principle thus follows from 
	Theorem \ref{thm:CPpde}.
	
\end{proof}

\begin{rem} Note that $\Theta(x) = \{F(x, \cdot, \cdot) \ge 0\} \cap \left( \cN
	\times \cS(N) \right) \cap \left( \R \times \cP \right)$, in other words
	\[
	(r, A) \in \Theta(x) \qquad \Leftrightarrow \qquad \min\{F(x, r, A), -r,
	\lambda_1(A)\} \ge 0.
	\]
Moreover, it can be easily checked that $\cQ$-admissible
	viscosity subsolutions and supersolutions of \eqref{cy} are equivalent to
	{\it standard} viscosity subsolutions and supersolutions (i.e. with no
	additional restrictions on the upper and lower test functions $\varphi$) of
	
	\[
	\min\{F(x, u(x), D^2u(x)), \ -u(x), \ \lambda_1(D^2u(x))\} = 0,
	\]
	which can be seen as a Bellman equation, or an obstacle problem for the
	fully non-linear equation \eqref{cy}. Indeed, classical solutions to
	$\min\{F(x, u, D^2u), -u, \lambda_1(D^2u)\} = 0$ are actually (convex)
	solutions to $\min\{F(x, u, D^2u), -u\} = 0$.
	
\end{rem}

Of course, the equation \eqref{cy} is just a prototype of equations for which the product structure is amenable to our methods. For example, one can obtain comparison principles for equations of the form $g(x, u) F(x, D^2 u) = h(x)$, assuming that
 $g(x, \cdot)$ is decreasing and and that $F(x, \cdot)$ is increasing, with
some strictness in at least one of the two variables, to guarantee the
validity of \eqref{UCF}. General examples of such $F$ can be found in
\cite[Section 5]{CP17}. In particular, our methods naturally cover more general
equations of the form
\begin{equation}\label{GdetL}
g(m(x) - u(x)) \det (D^2 u(x) + M(x)) = h(x),
\end{equation}
where $g,m,M$ are continuous functions, and $g(\cdot)$ is increasing
and positive on some open interval $(r_0, \infty)$. We stress that mere
continuity with respect to $x$ for $g, m, M$ is sufficient here, while
the application of general arguments in \cite{HL18} involving
jet-equivalence may require further regularity properties of data, as
Lipschitz continuity. See Definition 2.9 of \cite{HL18}. Note also that the equation \eqref{GdetL} is a
generalized version of an example discussed in \cite[Remark
5.10]{CP17}, where it is pointed out that condition \eqref{UCF} allows one to treat some cases in which the standard Crandall-Ishii-Lions
condition (see \cite[Condition (3.14)]{CIL92}) does not hold.

\begin{thm}\label{thm:CCCP} Suppose that $h \in C(\Omega)$ is non-negative, $m \in C(\Omega)$, $M \in C(\Omega; \cS(N))$ and $g \in C(\R)$ satisfies
	\begin{equation}\label{g_assumptions}
	\mbox{$g(\cdot)$ is increasing, $g(r_0) = 0$ and $g > 0$ on $(r_0, \infty)$ for some $r_0 \in \R$.}
	\end{equation} Then, the map $\Theta:
	\Omega \to \wp(\R \times \cS(N))$ defined by
	\begin{equation}\label{define_Theta}
	\Theta(x) := \{ (r,A) \in \Phi(x):  \  g(m(x) - r) \det (A + M(x))
	\geq h(x) \}
	\end{equation}
	where
	\begin{equation}\label{define_Phi}
	\Phi(x)  := \{(r,A) \in \R \times \cS(N): \ g(x) - r \ge r_0
	\text{ and } A +M(x) \ge 0\}
	\end{equation}
is a continuous proper elliptic map and defines a proper elliptic
branch of \eqref{GdetL}. Moreover, for any $\Phi$-admissible viscosity
subsolution $u$ and any $\Phi$-admissible viscosity supersolution $v$ of
\eqref{GdetL} (in the sense of Definition \ref{Vs_def}),
\[
\mbox{$u \leq v$ on $\partial \Omega \ \ \Rightarrow \ \ u \leq v$ in
	$\Omega$.}
\]
\end{thm}

\begin{proof} We apply Theorem \ref{thm:CPpde} in the constrained case, where one needs to check that the needed conditions hold. First, one easily checks that $\Phi$ defined by \eqref{define_Phi} a proper elliptic map; that is, each $\Phi(x)$ is a non-empty, closed proper subset of $\R \times \cS(N)$ which is $\cQ$-monotone, where we note that
		\begin{equation}\label{define_Phi2}
	\Phi(x)  := \{(r,A) \in \R \times \cS(N): \  r \leq m(x) - r_0
	\text{ and } A  \geq -M(x) \}
	\end{equation}
	so that $(r,A) \in \Phi(x)$ yields $(r + s, A + P) \in \Phi(x)$ for each $s \leq 0$ and $P \geq 0$. The proper elliptic map $\Phi$ is continuous, as one sees by using the local uniform continuity of $m$ and $M$ to show the local uniform continuity of $\Phi$ in the sense of \eqref{uusc2} in the characterization of Proposition \ref{UCHD}.
	Hence \eqref{TCP1} holds. 
	
	Next, the operator $F$ defined by
	\begin{equation}\label{define_F}
	F(x,r,A):= g(m(x) - r) \det (A + M(x)) - h(x)
	\end{equation}
	is $\cQ$-monotone in the sense \eqref{TCP4} since $(r,A) \in \Phi(x)$ and $(s,P) \in \cQ = \cN \times \cP$ yields
	$$
	m(x) - (r+s) \geq m(x) - r \geq r_0 \ \Rightarrow \ g(m(x) - (r+s)) \geq g(m(x) -r) \geq 0
	$$
	as $g$ is increasing and non-negative on $[r_0, +\infty)$ by \eqref{g_assumptions}, while
	$$
	A + P + M(x) \geq A + M(x) \geq 0 \ \Rightarrow \ \det(A + P + M(x)) \geq \det(A + M(x)) \geq 0.
	$$ 	
	
	The non-empty condition \eqref{TCP6} holds since for each $x \in \Omega$, the element
	$$(r_x,A_x) :=	\left(m(x) - r_0 - 1,  \left(\frac{h(x)}{m(r_0 + 1)} \right)^{1/N} I -	M(x) \right) \in \Phi(x)$$ 
	gives 
	$$
	F(x,r_x,A_x) := g(m(x) - r_x) \det (A_x + M(x)) - h(x) = m(r_0 + 1) \frac{h(x)}{m(r_0 + 1)} - h(x) = 0,
	$$ 
	where $m(r_0 + 1) > 0$ in view of the positivity assumption in \eqref{g_assumptions}.
	
	For the branch condition \eqref{TCP8}, since $\Phi(x) = (\cN + m(x)-r_0 ) \times (\cP + M(x))$, one has
	$$
	\partial \Phi(x) = (m(x)-r_0) \times (\cP - M(x)) \ \cup \ (\cN +
	g(x)-r_0 ) \times (\partial \cP + M(x)),
	$$
	and hence $F(x, \cdot, \cdot) \leq 0$ on $\partial \Phi(x)$ for each $x$.

	For the regularity property \eqref{TCP5}; that is, for $\Omega' \subset \subset \Omega$ and $\eta > 0$ arbitrary, there exists $\delta= \delta(\eta, \Omega') > 0$ such that
	\begin{equation}\label{TCP5_recall}
	F(y, r - \eta, A + \eta I) \geq F(x, r, A) \quad \forall (r, A) \in \Phi(x), \forall x,y \in \Omega' \ {\rm with \ } |x - y| < \delta,
	\end{equation}
	one makes use of the local uniform continuity of $m,M$ and $g$ together with the monotonicity of $g$ and $\det$ on $\Phi(x)$. Indeed,
	pick  $\delta = \delta(\eta, \Omega')>0$ so that for each $x,y \in \Omega'$ with $|x-y| < \delta$ one has:
	\begin{equation}\label{delta_choices}
	m(y) - m(x) + \frac{\eta}{2} \geq 0, \ M(y) - M(x) + \frac \eta 2 \geq 0, \ 
	h(x) - h(y) + g(r_0+\eta/2)\left(\frac\eta2\right)^N \geq 0,
	\end{equation} 
where $g(r_0 + \eta/2) > 0$ by \eqref{g_assumptions}. Since $g$ is increasing on $[r_0, +\infty)$, the first condition in \eqref{delta_choices} together with $m(x) - r \geq r_0$ if $(r,A) \in \Phi(x)$ yields
\begin{equation}\label{est1}
g(m(y) - r + \eta) \geq g(m(x) - r + \eta/2) \geq g(r_0 + \eta/2) > 0.
\end{equation}
Similarly, since $\det$ is increasing on $\cP$, the second condition in \eqref{delta_choices} together with $A + M(x) \in \cP$ for $(r,A) \in \Phi(x)$ yields
\begin{equation}\label{est2}
\det(A + \eta I + M(y)) \geq \det(A + (\eta/2) I + M(x)) \geq \det(A + M(x)) \geq 0.
\end{equation}
Using \eqref{est1}, \eqref{est2} and the inequality $\det(A + B) \geq \det(A) = \det(B)$ for $A, B \geq 0$, one finds
\begin{multline*}
	F(y, r - \eta, A + \eta I) = g(m(y) - r + \eta)\det(A + \eta I + M(y)) - h(y) \geq\\
	g(m(y) - r + \eta) \left[ \det\left(A + \frac\eta2 I + M(y)\right) +   \left(\frac\eta2\right)^N \right] - h(y) \geq \\
	g(m(x) - r + \eta/2)\det\left(A + M(x) \right) + g(r_0 + \eta/2) \left(\frac\eta2\right)^N  - h(y)
	\\ \geq g(m(x) - r )\det\left(A + M(x) \right) - h(x) = F(x, r, A),
	\end{multline*}
	where we have also used the third condition in \eqref{delta_choices} in the last inequality.
	
	Finally, the non-degeneracy condition \eqref{TCP7} follows from the structure of $F(x, \cdot, \cdot)$ on $\Theta(x)$. Indeed, if $(r,A) \in [\Theta(x)]^{\circ}$, one must have both
	\begin{equation}\label{r_condition}
	m(x) - r > r_0 \ \text{and hence} \ g(m(x) - r) > 0 \ \text{by} \ \eqref{g_assumptions}
	\end{equation}
	and 
	\begin{equation}\label{A_condition}
A + M(x) \in \cP^{\circ} \ \text{and hence} \ \det(A + M(x)) > 0. 
	\end{equation}
	Consequently, in the equation
	\begin{equation}\label{F=0}
	F(x,r,A) = g(m(x) - r) \det(A + M(x)) - h(x) = 0, 
	\end{equation}
	if $(r,A) \in [\Theta(x)]^{\circ}$, then one must have $h(x) > 0$. however, in this case by the positivity in \eqref{r_condition} and the strict monotonicity of $\det$ on $\cP^{\circ}$ one cannot preserve $F(x, \cdot, \cdot) \geq 0$ in a neighborhood of $(r,A) \in [\Theta(x)]^{\circ}$.

\end{proof}

\subsection{Comparison in the unconstrained case.} 

As a final illustration of our method, we will prove a new comparison principle for the inhomogeneous {\em special
Lagrangian potential equation} 
\begin{equation}\label{sleq}
\sum_{i = 1}^N \arctan \big( \lambda_i (D^2 u(x)) \big) =
h(x).
\end{equation}
where
\begin{equation}\label{evals}
\lambda_1(A) \leq \lambda_2(A) \leq \cdots \leq \lambda_N(A)
\end{equation}
are the ordered eigenvalues of $A \in \cS(N)$ and $h \in C(\Omega)$. As noted in the introduction, while this equation is proper elliptic on all of $\R \times \cS(N)$, its treatment is delicate due to the degeneracies when the operator
$G :\cS(N) \to \cI = (-N\pi/2, N\pi/2)$ defined by
\begin{equation}\label{slpo}
G(A) := \sum_{i = 1}^N \arctan \big( \lambda_i (A) \big), \ \ A \in S(N)
\end{equation}
takes on one of the {\em special phase values}
\begin{equation}\label{SPV6}
\theta_k:= (N - 2k)\pi/2 \ \ \text{for} \  k = 1, \ldots N - 1,
\end{equation}
which determine the {\em phase intervals}
\begin{equation}\label{I_k6}
\cI_k:= \left( (N - 2k) \frac{\pi}{2}, (N - 2(k - 1)) \frac{\pi}{2} \right) \ \ \text{with} \ \ k = 1, \ldots N.
\end{equation}

As discussed in the introduction, we will make a contribution to the following Open Question (page 23 of \cite{HL19}): {\em does the comparison principle hold for each continuous phase function $h$ taking values in $\cI$?} For $h$ taking values in the top phase interval $\cI_1$, this is known (see \cite{DDT19} or \cite{HL19}). We will show that the comparison principle holds if $h$ takes values in any one of the phase intervals $\cI_k$ with $k \in \{ 1, \ldots, N\}$. We will also show that our method breaks down if $h$ takes on one of the special phase values $\theta_k$.

We begin by embedding the PDE \eqref{sleq} into its natural potential theoretic framework. For the pure second order operator \eqref{slpo}
consider the map $\Theta: \Omega \to \wp(\R \times \cS(N))$ defined by 
\begin{equation}\label{thetasl}
\Theta(x) := \{ (r,A) \in \R \times \cS(N) :  \  G(A) \ge h(x) \}.
\end{equation}
One easily checks that $\Theta$ is a proper elliptic map and defines a proper elliptic branch of \eqref{sleq}. Indeed, each $\Theta(x)$ is closed by the continuity of $G$ and $h$, $\Theta(x)$ is non-empty since
$$
	(r_x,A_x) \in \Theta(x) \ \ \text{for each} \ r_x \in \R \ \text{and} \ A_x = \tan {\left( \frac{h(x)}{N} \right)} I,
$$
and $\Theta(x) \subsetneq \R \times \cS(N)$ since one easily finds $(r,A)$ such that $G(A) - h(x) < 0$ by using the monotonicity of $F$ on all of $\cS(N)$. The operator $F(x,r,A):= G(A) - h(x)$ is clearly proper elliptic on $\R \times \cS(N)$ so that $\Theta(x)$ is $\cQ$-monotone. Finally, $\Theta$ defines a branch of \eqref{sleq} since
\begin{equation}\label{sleq_branch}
	\partial \Theta(x) = \{ (r,A) \in \R \times \cS(N): \ G(A) = h(x)\}.
\end{equation}

Having embedded the PDE problem into its natural potential theoretic framework, the main point is to verify the continuity of the map $\Theta$. This is delicate since the operator $G$ degenerates on level sets $\{ A \in \cS(N): \ G(A) = \theta_k \}$ for each $k = 1, \ldots, N -1$. We will show that $\Theta$ is continuous if $h$ avoids the special phase values $\theta_k$ and hence the comparison principle follows from Theorem \ref{thm:CP}. The key idea in the proof is that if $h \in C(\Omega)$ avoids each special phase value, then $h$ maps compact subsets of $\Omega$ into a compact subset $\Sigma$ of some open phase interval $\cI_k$, which yields a locally uniform bound on at least one eigenvalue of $A \in G^{-1}(\Sigma)$ (see Lemma \ref{lem:slpe} below). It is perhaps instructive to first show how continuity fails nearby a point $x_0$ where $h$ (non constant) does take on a special phase value, and hence our method breaks down.

\begin{prop}[Failure of continuity]\label{prop:failure} Let $h \in C(\Omega)$ and suppose that there exists a convergent sequence $\{x_n\}_{n \in \N} \subset \Omega$ with limit $x_0 \in \Omega$ such that for some $k = 1, \ldots , N -1$
\begin{equation}\label{failure1}
h(x_0) = \theta_k := (N -2k) \frac{\pi}{2}
\end{equation}
and	
\begin{equation}\label{failure2}
\text{either} \ h(x_n) > \theta_k \ \text{for every} \ n \in \N \ \ \text{or} \ \ h(x_n) < \theta_k \ \text{for every} \ n \in \N .
\end{equation}
Then the proper elliptic map $\Theta$ is not continuous.
\end{prop}

\begin{proof}
	By Proposition \ref{proper_ell_map}(c) with $\Omega'$ any neighborhood of $x_0$ and $\eta = 1$, it suffices to show that either
	\begin{equation}\label{fail1}
		\Theta(x_0) + (-1, I) \not\subset \Theta(x_n) \ \ \text{for each} \ n \in \N.
		\end{equation}
		or 
	\begin{equation}\label{fail0}
	\Theta(x_n) + (-1, I) \not\subset \Theta(x_0) \ \ \text{for each} \ n \in \N.
	\end{equation}	
	We treat first the case $h(x_n) > \theta_k$.  By the definition of $\Theta$, in order to show \eqref{fail1} it suffices to exhibit a sequence $\{A_n\} \subset \cS(N)$ such that
	\begin{equation}\label{fail2}
	G(A_n) \geq h(x_0) = \theta_k \ \ \text{and} \ \ G(A_n + I) < h(x_n).
	\end{equation}
	One such sequence is provided by the block diagonal matrices
	\begin{equation}\label{fail3}
	A_n:= \left[ \begin{array}{cc} -a_n I_{k} & 0 \\ 0 &  b_n I_{N - k} \end{array} \right], \ \ \ n \in \N,
	\end{equation}
	where $I_k \in \cS(k)$ is the identity matrix and  $a_n, b_n > 0$ are to be chosen  suitably large, but are constrained to satisfy
		\begin{equation}\label{fail4}
	G(A_n) = (N - k) \arctan{(b_n)} - k \arctan{(a_n)} = \theta_k,
	\end{equation}
	so that the first condition in \eqref{fail2} holds. By making use of the mean value theorem one one finds $\xi_n \in (-a_n, -a_n + 1)$ and $\eta_n \in (b_n, b_n + 1)$ such that
	\begin{eqnarray*}
		G(A_n + I) & = &  (N - k) \arctan{(b_n + 1)} - k \arctan{(-a_n + 1)}  \\
				& = & (N - k)  \arctan{(b_n)} + \frac{1}{1 + \eta_n^2} +   k \arctan{(-a_n)} + \frac{1}{1 + \xi_n^2},  
		\end{eqnarray*}
which by \eqref{fail4} yields
\begin{equation}\label{fail5}
		G(A_n + I) =	\theta_k +  \frac{1}{1 + \eta_n^2} + \frac{1}{1 + \xi_n^2}.
		\end{equation}
If one sends $a_n$ to $+\infty$ then by \eqref{fail4} and the definition of $\theta_k$ one must have 
$$
\arctan{(b_n)} \to \frac{1}{N-k} \left( \theta_k + \frac{k \pi}{2} \right) = \frac{\pi}{2}
$$ 
and hence also $b_n$ goes to $+ \infty$. This also forces $\xi_n^2$ and $\eta_n^2$ to infinity. Hence by choosing $a_n$ large enough one can make the right hand side of \eqref{fail5} smaller than $h(x_n)$, where $h(x_n) > \theta_k$ by hypothesis.

Similarly, in the case $h(x_n) < \theta_k$, one can show that \eqref{fail0} holds by exhibiting a sequence $\{A_n\}_{n \in \N} \subset \cS(N)$ such that
	\begin{equation}\label{fail6}
G(A_n) \geq h(x_n)  \ \ \text{and} \ \ G(A_n + I) <  \theta_k = h(x_0).
\end{equation}
Again we take $A_n$ of the form \eqref{fail3} with the constraint \eqref{fail4} replaced by
\begin{equation}\label{fail7}
G(A_n) = (N - k) \arctan{(b_n)} - k \arctan{(a_n)} = h(x_n) < \theta_k,
\end{equation}
so that the first condition in \eqref{fail6} holds. By choosing $a_n$ and $b_n$ large enough, the same mean value argument used above gives
$$
G(A_n + I) = h(x_n) +  \frac{1}{1 + \eta_n^2} + \frac{1}{1 + \xi_n^2} < \theta_k
$$
so that the second condition in \eqref{fail6} holds, which completes the proof.
\end{proof}

On the other hand, if $h$ avoids the special phase values, then the proper elliptic map $\Theta$ is continuous.

\begin{thm}\label{thm:CSLPE} Suppose that $h \in C(\Omega)$ satisfies
	\begin{equation}\label{h_condition}
	h(\Omega) \subset \cI_k:= \left( (N - 2k) \frac{\pi}{2}, (N - 2(k - 1)) \frac{\pi}{2} \right) \ \ \text{for some fixed} \ k \in \{1, \ldots, N\}.
	\end{equation}
Then, the proper elliptic map
	$\Theta$ defined in \eqref{thetasl} is continuous, and defines a proper
	elliptic branch of \eqref{sleq}. Moreover, for any viscosity subsolution
	$u$ and any viscosity supersolution $v$ of \eqref{sleq},
	\[
	\mbox{$u \leq v$ on $\partial \Omega \ \ \Rightarrow \ \ u \leq v$ in
		$\Omega$.}
	\]
\end{thm}

\begin{proof} We have already discussed the claims that $\Theta$ is 
	proper elliptic and defines a proper elliptic branch of \eqref{sleq}.
	Moreover, the branch condition \eqref{sleq_branch} is precisely the
	needed non-degeneracy condition \eqref{Theta_ND} which ensures the
	correspondence principle between viscosity supersolutions and
	$\Theta$-superharmonic functions (as discussed in Remark \ref{rem:UC_Case}).
	
	Hence, in order to have the comparison principle for
	$\Theta$-sub/superharmonic functions, it suffices to verify that
	$\Theta$ is a continuous map (in order to apply the potential theoretic
	comparison Theorem \ref{thm:CP}). 
	
	To prove that $\Theta$ is indeed continuous, one can argue as follows.
	By Remark \ref{rem:LUC} and Proposition \ref{UCHD}(c), the proper elliptic
	map $\Theta$ will be continuous on $\Omega$ if for each $\Omega' \subset
	\subset \Omega$ and each $\eta > 0$ there exists $\delta(\eta, \Omega')
	> 0$ such that for each $x,y \in \Omega'$ with $|x-y| < \delta$
	\begin{equation}\label{LUC_recall1}
	(r, A) \in \Theta(x) \ \ \Rightarrow \ \ (r + \eta, A + \eta I) \in
	\Theta(y),
	\end{equation}
	which in terms of the special Lagrangian potential operator $G$ requires showing that for each $\eta > 0$ and each pair $x,y \in \Omega'$ with $|x-y| < \delta$ one has the implication
	\begin{equation}\label{LUC_recall2}
	G(A) - h(x) \geq 0 \ \ \Rightarrow \ \ G(A + \eta I) - h(y) \geq 0.
	\end{equation}
	Notice that the continuity of $h$ and condition \eqref{h_condition} shows that $h\big(\overline {\Omega'} \big)$ is a compact subset of the open set $\cI_k$ and hence  
	\begin{equation}\label{hss1}
	h\big(\overline {\Omega'} \big) \subset [\alpha, \beta] \subset (\theta_{k+1}, \theta_k) = \cI_k \ \  \text{for some} \ \alpha, \beta \in \cI_k.
	\end{equation}
Hence there exists $\veps > 0$ (depending on $\Omega'$ and $h$) such
\begin{equation}\label{hss2}
h\big(\overline {\Omega'} \big) + \gamma \subset [\alpha, \beta + \veps] \subset (\theta_{k+1}, \theta_k) = \cI_k \ \  \text{for all} \ \gamma \in [0, \veps].
\end{equation}
There are two cases to consider for the pair $(x,A)$ in \eqref{LUC_recall2}; namely,
	\begin{equation}\label{G_cases}
	G(A) \geq h(x) + \veps \ \ \text{and} \ \ h(x) \leq G(A) < h(x) +
	\veps.
	\end{equation}
	
	In the first case of \eqref{G_cases}, using the monotonicity of $G$
	and the uniform continuity of $h$ on $\overline{\Omega'}$, one has
	$$
	G(A + \eta I) - h(y) \ge G(A)  - h(y) \ge h(x) - h(y) +
	\varepsilon \geq 0
	$$
	for each $x,y \in \Omega'$ with $|x-y| \le \delta =
	\delta(\Omega')$, but independent of $\eta > 0$.
	
	For the second case of \eqref{G_cases}, we make use of the
	following fact. See Figure \ref{fig} which represents the case of dimension $N = 2$ for the interval $\cI_1 = (0, \pi)$.
	
	\begin{lem}\label{lem:slpe} For any $\Sigma$ compact in $\cI_k$ open, the set
		$G^{-1}(\Sigma) = \{A \in \cS(N): \ G(A) \in \Sigma\}$ satisfies the
		following property: there exists $C = C(\Sigma) > 0$ such that
	\begin{equation}\label{C_condition}
	\mbox{{\em if} $A \in G^{-1}(\Sigma)$,  \ {\em then}
		$|\lambda_j(A)| \le C$ \ {\em for
			some} $j= j_A \in \{ 1, \ldots, N \}$. }
	\end{equation}
	\end{lem}
	
	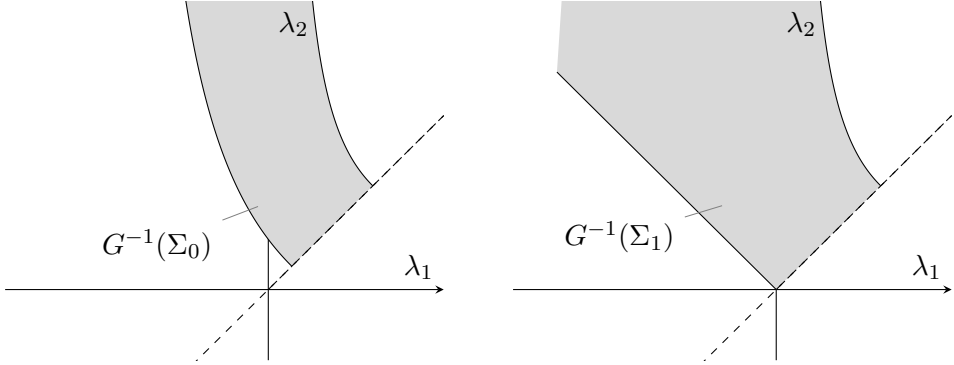
\begin{figure} 
		\centering
		\begin{tikzpicture}
		\begin{axis}[
		width=.5\textwidth, axis lines=middle, axis equal,
		xmin=-3,xmax=2,ymin=0,ymax=2.5,
		xtick=\empty, ytick=\empty,
		xlabel={$\lambda_1$},
		ylabel={$\lambda_2$}
		]
		\path[name path=topp] (axis cs:-3,2) -- (axis cs:3,2);
		\addplot[samples=300, domain=-1:tan(30/2), name path=A]
		{tan(30-atan(x))};
		\addplot[dashed, samples=300, domain=-1:2, name path=B] {x};
		\addplot[samples=300, domain=.2:tan(50), name path=C]
		{tan(100-atan(x))};
		\addplot[gray!30] fill between[of=A and C];
		\addplot[dashed, samples=300, domain=0:2, name path=B] {x};
		\node[pin=200:{$G^{-1}(\Sigma_0)$}] at (axis cs:0,1) {};
		\end{axis}
		\end{tikzpicture} \qquad
		\begin{tikzpicture}
		\begin{axis}[
		width=.5\textwidth,axis lines=middle, axis equal,
		xmin=-3,xmax=2,ymin=0,ymax=2.5,
		xtick=\empty, ytick=\empty,
		xlabel={$\lambda_1$},
		ylabel={$\lambda_2$}
		]
		\path[name path=topp] (axis cs:-3,2) -- (axis cs:3,2);
		\addplot[samples=300, domain=-2.5:0, name path=A] {-x};
		\addplot[dashed, samples=300, domain=-1:2, name path=B] {x};
		\addplot[samples=300, domain=.2:tan(50), name path=C]
		{tan(100-atan(x))};
		\addplot[gray!30] fill between[of=A and C];
		\addplot[dashed, samples=300, domain=0:2, name path=B] {x};
		\node[pin=190:{$G^{-1}(\Sigma_1)$}] at (axis cs:-.5,1) {};
		\end{axis}
		\end{tikzpicture}
		\caption{\footnotesize Gray regions are examples of $G^{-1}(\Sigma)$ in
			the half-plane $\{\lambda_2 \ge \lambda_1\} \subset \cS(2)$. On the left
			$\Sigma_0 = [\pi/6, \pi/2] \subset\subset \cI_1 = (0, \pi)$, while on the right
			$\Sigma_1 = [0, \pi/2] \subsetneq \overline{\cI}_1$. Note that in the first case
			$\lambda_1(A)$ is bounded uniformly for $A \in G^{-1}(\Sigma_0)$, while
			in the latter case one can find sequences $\{A_n\} \subset
			G^{-1}(\Sigma_1)$ such that $\lambda_1(A_n) \to -\infty$ and
			$\lambda_2(A_n) \to +\infty$ as $n \to \infty$.}\label{fig}
	\end{figure}
	
	 We postpone the proof of the Lemma, proceeding with the analysis of the remaining case $h(x) \le
	G(A) \le h(x) + \varepsilon$; that is,
	$$
	A \in G^{-1}\big([h(x), h(x) + \varepsilon]\big).
	$$
	By \eqref{hss2}, we have $[h(x), h(x) + \varepsilon] \subset \bigcup_{\gamma \in [0,
		\varepsilon]}\{h\big(\overline {\Omega'} \big) + \gamma \} =: \Sigma$ is
	compact in $\cI_k$ open and hence by Lemma \ref{lem:slpe} we have
	$\lambda_{j_A}(A) \leq C = C(\Sigma)$ for
	some $j_A \in \{1, \ldots, N\}$ . The mean
	value theorem then implies
	\begin{equation}\label{arctan}
	\arctan \big( \lambda_{j_A} (A + \eta I) \big) = \arctan \big(
	\lambda_{j_A}
	(A) + \eta \big) = \arctan \big( \lambda_{j_A} (A) \big) + \frac{1}{1 +
		\xi^2}\eta
	\end{equation}
	for some $\xi \in (\lambda_{j_A}, \lambda_{j_A} + \eta)$. We have
	$|\xi| \leq C + \eta$ by \eqref{C_condition} and hence for $\eta \leq
	\eta^*:= C$
	we have $(1 + \xi^2)^{-1} \geq (1 + 4 C^2)^{-1}$. Therefore, for
	$\eta \in (0, \eta^*)$ by the monotonicity of $\arctan$ and
	\eqref{arctan} we have
	\begin{eqnarray*}
		G(A + \eta I) - h(y) & \geq &  \arctan \big( \lambda_{j_A}
		(A+ \eta I)   \big) + \sum_{j \neq j_A} \arctan \big( \lambda_{j}
		(A)   \big) - h(y) \\
		& \geq & G(A) + \frac{\eta}{1 + 4C^2}- h(y) \geq G(A) - h(x)
		\geq 0
	\end{eqnarray*}
	for all $x,y \in \Omega'$ with $|x-y| < \delta(\eta, \Omega')$ to ensure
	$|h(x) - h(y)| < (1 + 4C^2)\eta$, which we can do by the local uniform
	continuity of $h$. The same $\delta$ also works for $\eta > \eta^*$ as
	$G$ is increasing. This completes the proof of the continuity
	\eqref{LUC_recall1} modulo the proof of Lemma \ref{lem:slpe}. 
\end{proof}

\begin{proof}[Proof of Lemma \ref{lem:slpe}] We argue by
	contradiction. Let $\{A_n\}_{n \in \N} \subset G^{-1}(\Sigma)$ be a
	sequence such that
	$|\lambda_j(A_n)| \to +\infty$ for all $j = 1, \ldots, N$. The set
	of all possible accumulation points of $\{ G(A_n) \}_{n \in \N}$ is
	$\left\{ (N - 2k)\pi/2: \ k = 0, 1, \ldots, N \right\}$, which
	correspond to subsequences with
	$$
	\mbox{ $\lambda_1(A_{n_\ell}), \ldots, \lambda_k(A_{n_\ell}) \to -
		\infty$ \ \ and \ \ $\lambda_{k+1}(A_{n_\ell}),
		\ldots, \lambda_N(A_{n_\ell}) \to + \infty$}.
	$$
	Since such accumulation points do not belong to $\cI_k \supset \Sigma$, they
	also do not belong to $\Sigma$, and therefore $G(A_n) \notin \Sigma$ for
	$n$ large enough, which contradicts $A_n \in G^{-1}(\Sigma)$.
	\end{proof}

We conclude with a pair or remarks concerning possible generalizations.

\begin{rem}\label{rem:SLPE}
Note that \eqref{sleq} is $u$-independent. It is still interesting to
observe how non-degeneracy properties of the operator $G$ and regularity
of the inhomogeneous term affect continuity of the associated elliptic
map $\Theta$. We also point out that our comparison results could be easily
extended to cover more general $u$-dependent equations of the form
\begin{equation}
G(D^2 u(x)) = h(x, u(x)),
\end{equation}
under the assumption that $h$ is continuous, monotone in the second
variable, and satisfies $h(\Omega \times \R) \subset \cI_k$ for some $k$.
\end{rem}

\begin{rem}\label{cut_example} The comparison principle stated in Theorem
	\ref{comparison_cut}, that uses a slightly relaxed version of Hausdorff
	continuity \eqref{cont_cut}, might be useful in situations where terms
	of the form $g(x, u)$ appear in the equation. For $c \in C(\overline
	\Omega)$, $c \ge 0$, consider for example the proper elliptic map
	\[
	\Theta(x) := \{ (r,A) \in \R \times \cS(N) :  \  {\rm tr}(A)  - c(x) r
	\ge 0 \},
	\]
	which defines a proper elliptic branch of the linear PDE
	\[
	\Delta u(x) - c(x) u(x) = 0.
	\]
	To check Hausdorff continuity, for any given $\eta > 0$, and $(r, A) \in
	\Theta(x)$,
	\begin{multline*}
	{\rm tr}(A + \eta I)  - c(y) (r - \eta) =  {\rm tr}(A)  - c(x)r + [c(y)
	- c(x)]r + \eta (N + c(y)) \\
	\ge [c(y) - c(x)]r + \eta (N + c(y)).
	\end{multline*}
	Then, one has for general $c$ that $[c(y) - c(x)]r + \eta (N + c(y)) \ge
	0$ for $x$ close to $y$ , that is $(r - \eta, A + \eta I) \in
	\Theta(y)$, only if $r$ lies in a bounded subset of $\R$. This is
	precisely what \eqref{cont_cut} requires, while $\Theta$ would fail to
	satisfy the stronger Hausdorff continuity (unless $c$ is constant over
	$\Omega$).
\end{rem}

\begin{ack}
	The authors wish to thank Reese Harvey for many elucidating discussions and constructive criticism, especially concerning the application given here to the special Lagrangian potential equation. 
\end{ack}


\begin{thebibliography}{99}
	\bibitem{AC84} J-P.\ Aubin and A.\ Cellina, {\em Differential Inclusions: Set Valued Maps and Viability Theory}, Volume 264 of {\em Grundlehren der mathematischen Wissenschaften}, Springer-Verlag, Berlin, 1984.
	\bibitem{BM06} M.\ Bardi and P.\ Mannucci, {\em On the Dirichlet problem for non-totally degenerate fully nonlinear elliptic equations}, Commun.\ Pure Appl.\ Anal.\ {\bf 73} \ (2006), 709--731.
	\bibitem{BB01} G.\ Barles and J.\ Busca, {\em Existence and comparison results for fully nonlinear degenerate  elliptic equations without zeroth-order term}, Comm.\ Partial Differential Equations\ {\bf 26} \ (2001), 2323--2337.
		\bibitem{BP19} I.\ Birindelli and K.\ R.\ Payne, {\em Principal eigenvalues for $k$-Hessian operators by maximum principle methods}, {\tt arxiv.org/abs/1912.09226v1}, 19 December 2019, 1--42.
	\bibitem{BBI01} D.\ Burago, Y.\ Burago and S.\ Ivanov,
		{\em A Course in Metric Geometry}, Volume 33 of {\em Graduate Studies in Mathematics}, American Mathematical Society, Providence, RI, 2001.
	\bibitem{CY86} S.Y.\ Cheng and S.T.\ Yau, {\em Complete affine hypersurfaces. Part I. The completeness of affine metrics}, Comm.\ Pure Appl.\ Math.\ {\bf 39} \ (1986), 839--866.
	\bibitem{CHLP19} M.\ Cirant, F.R.\ Harvey, H.B.\ Lawson, Jr. and K.R.\ Payne, {\em Comparison principles for constant coefficient nonlinear second order equations}, preprint, (2020).
	\bibitem{CP17} M.\ Cirant and K.R.\ Payne, {\em On viscosity solutions to the Dirichlet problem for elliptic branches of nonhomogeneous fully nonlinear equation}, Publ.\ Mat.\ {\bf 61}\ (2017), 529--575.
	\bibitem{CIL92} M.G.\ Crandall, H.\ Ishii and P-L.\ Lions, {\em User's guide to viscosity solutions of second order partial differential equations}, Bull.\ Amer.\ Math.\ Soc.\ {\bf 27}\ (1992), 1--67.
	\bibitem{DDT19} S.\ Dinew, H-S.\ Do and T.D.\ T\^{o}, {\em A viscosity approach to the Dirichlet problem for degenerate complex Hessian-type equations}, Anal.\ PDE \ {\bf 12}\ (2019), 505-–535.
\bibitem{HL82} F.R.\ Harvey and H.B.\ Lawson, Jr., {\em Calibrated geometries}, Acta Math.\ {\bf148} \ (1982), 47--157.
\bibitem{HL09} F.R.\ Harvey and H.B.\ Lawson, Jr., {\em Dirichlet duality and the nonlinear Dirichlet problem}, Comm.\ Pure Appl.\ Math.\ {\bf 62}\ (2009), 396--443.
\bibitem{HL10} F.R.\ Harvey and H.B. Lawson, Jr., {\em Hyperbolic polynomials and the Dirichlet problem}, {\tt arXiv:0912.5220v2}, 19 March 2010, 1--48.
\bibitem{HL11} F.R.\ Harvey and H.B.\ Lawson, Jr., {\em Dirichlet duality and the nonlinear Dirichlet problem on Riemannian manifolds}, J.\ Differential Geom.\ {\bf 88}\ (2011), 395--482.
\bibitem{HL13} F.R.\ Harvey and H.B. Lawson, Jr., {\em G\aa rding's theory of hyperbolic polynomials}, Comm.\ Pure Appl.\ Math.\ {\bf 66}\ (2013), 1102--1128.
\bibitem{HL18} F.R.\ Harvey and H.B. Lawson, Jr., {\em The inhomogeneous Dirichlet Problem for natural operators on manifolds}, Ann.\ Inst.\ Fourier (Grenoble), to appear {\tt arXiv:1805.111213v1}, 28 May 2018, 1--44.
\bibitem{HL19} F.R.\ Harvey and H.B.\ Lawson, Jr., {\em Pseudoconvexity for the special Lagrangian potential equation}, preprint, 2019, pages 1--37.
\bibitem{KK07} B.\ Kawohl and N.\ Kutev, {\em Comparison principle for viscosity solutions of fully nonlinear, degenerate elliptic equations}, Comm.\ Partial Differential Equations, {\bf 32}\ (2007), 1209--1224.
	\bibitem{Kv95} N.V.\ Krylov, {\em On the general notion of fully nonlinear second-order elliptic equations}, Trans.\ Amer.\ Math.\ Soc.\ {\bf 347} (1995), 857--895.
	\bibitem{Sl84} Z.\ Slodkowski, {\em The Bremermann-Dirichlet problem for q-plurisubharmonic functions}, Ann.\ Scuola Norm.\ Sup.\ Pisa Cl.\ Sci.\ (4), {\bf 11}\ (1984), 303--326.

\end{thebibliography}
\end{document}